\documentclass[graybox]{svmult}

\usepackage{type1cm}        
\usepackage{makeidx}         
\usepackage{graphicx}        
\usepackage{multicol}        
\usepackage[bottom]{footmisc}

\usepackage{newtxtext}       %
\usepackage[varvw]{newtxmath}       

\usepackage{bm}

\makeindex             

\usepackage{amsmath,amsbsy,amssymb,amsthm}

\makeatletter
\newcommand{\doublewidetilde}[1]{{%
  \mathpalette\double@widetilde{#1}%
}}
\newcommand{\double@widetilde}[2]{%
  \sbox\z@{$\m@th#1\widetilde{#2}$}%
  \ht\z@=.9\ht\z@
  \widetilde{\box\z@}%
}
\makeatother

\usepackage{bbm,bm,mathrsfs}

\usepackage{tikz}
\usepackage{enumitem}
\usepackage{graphicx}

\usepackage{mathtools}
\usepackage{soul}
\usepackage{xcolor}
\usepackage[hidelinks]{hyperref}

\newcommand{\Var}{\operatorname{Var}}
\newcommand{\Bias}{\operatorname{Bias}}
\newcommand{\MSE}{\operatorname{MSE}}
\newcommand{\AMSE}{\operatorname{AMSE}}
\newcommand{\MISE}{\operatorname{MISE}}

\newcommand{\Prob}{{\mathbb{P}}}
\newcommand{\Expec}{{\mathbb{E}}}

\DeclarePairedDelimiter{\norm}{\lVert}{\rVert}

\def\to{\rightarrow}

\def\Ec{{\mathcal E}}

\def\Eb{{\mathbb E}}

\def\Nb{{\mathbb N}}

\def\Rb{{\mathbb R}}

\def\E{{\mathbf E}}

\def\I{{\mathbf I}}

\def\X{{\mathbf X}}

\def\x{{\mathbf x}}

\def\z{{\mathbf z}}

\def\mubf{ {\bm \mu}}
\def\Sigmabf{ {\bm \Sigma}}

\newcommand{\psiax}{{ \psi_a(\norm{\x}^2) }}

\newcommand{\psiaxi}{{ \psi_a(\xi) }}
\newcommand{\psiaXi}[1]{{ \psi_a(\xi_{#1}) }}

\newcommand{\psiaxiprime}{{ \psi_a'(\xi) }}
\newcommand{\psiaxisecond}{{ \psi_a''(\xi) }}
\newcommand{\psiaxithird}{{ \psi_a'''(\xi) }}

\newcommand{\psixi}{{ \psi(\xi) }}

\newcommand{\gderivk}{{ g^{(k)} }}
\newcommand{\gderivi}{{ g^{(i)} }}

\newcommand{\gxi}{{ g(\xi) }}
\newcommand{\gxiprime}{{ g'(\xi) }}
\newcommand{\gxisecond}{{ g''(\xi) }}
\newcommand{\gxiderivk}{{ g^{(k)}(\xi) }}
\newcommand{\gxiderivi}{{ g^{(i)}(\xi) }}

\newcommand{\tildegxi}{{\, \widetilde{g}(\xi) }}
\newcommand{\tildegxiprime}{{\, \widetilde{g}'(\xi) }}
\newcommand{\tildegxisecond}{{\, \widetilde{g}''(\xi) }}

\newcommand{\zaxi}{{ z_a(\xi) }}
\newcommand{\zaxiprime}{{ z_a'(\xi) }}
\newcommand{\zaxisecond}{{ z_a''(\xi) }}

\newcommand{\hopt}{{ h_{\textnormal{opt}} }}
\newcommand{\aopt}{{ a_{\textnormal{opt}} }}
\newcommand{\Aopt}{{ A_{\textnormal{opt}} }}

\usepackage[ruled]{algorithm2e}

\makeatletter
\algocf@newcmdside@kobe{ForEvery@for}{%
    \KwSty{For every\hphantom{f}}%
    \ifArgumentEmpty{#1}\relax{ #1}%
    #2
    \KwSty{do}%
}
\algocf@newcmdside@kobe{ForEvery@then}{%
    \ifArgumentEmpty{#1}\relax{ #1}%
    \algocf@block{#2}{end}{#3}%
    \par
}
\newcommand\ForEvery[2]{%
    \ForEvery@for{#1}%
    \ForEvery@then{#2}%
}
\makeatother

\DeclareFontFamily{U}{mathx}{}
\DeclareFontShape{U}{mathx}{m}{n}{<-> mathx10}{}
\DeclareSymbolFont{mathx}{U}{mathx}{m}{n}
\DeclareMathAccent{\widecheck}{0}{mathx}{"71}

\usepackage{placeins}  
\usepackage{amsmath}

\newtheorem{conj}{Conjecture}[section]
\newtheorem{thm}[conj]{\bf Theorem}

\newtheorem{cor}[conj]{\bf Corollary}
\newtheorem{prop}[conj]{\bf Proposition}

\begin{document}

\title*{On the choice of the two tuning parameters for nonparametric estimation of an elliptical distribution generator}
\titlerunning{Tuning parameters for estimation of elliptical generators}
\author{Victor Ryan and \\ Alexis Derumigny\orcidID{0000-0002-6163-8097}}
\institute{{Department of Applied Mathematics, Delft University of Technology
\at Mekelweg 4, 2628 CD, Delft, Netherlands, 
\email{a.f.f.derumigny@tudelft.nl}
}
}
%
%
\maketitle

\abstract*{{Elliptical distributions are a simple and flexible class of distributions that depend on a one-dimensional function, called the density generator.
In this article, we study the non-parametric estimator of this generator that was introduced by Liebscher (2005). This estimator depends on two tuning parameters: a bandwidth $h$ -- as usual in kernel smoothing -- and an additional parameter $a$ that control the behavior near the center of the distribution.
We give an explicit expression for the asymptotic MSE at a point $x$, and derive explicit expressions for the optimal tuning parameters $h$ and $a$.
Estimation of the derivatives of the generator is also discussed.
A simulation study shows the performance of the new methods.}}

\abstract{{Elliptical distributions are a simple and flexible class of distributions that depend on a one-dimensional function, called the density generator.
In this article, we study the non-parametric estimator of this generator that was introduced by Liebscher~\cite{liebscher2005_semiparametric}. This estimator depends on two tuning parameters: a bandwidth $h$ -- as usual in kernel smoothing -- and an additional parameter $a$ that control the behavior near the center of the distribution.
We give an explicit expression for the asymptotic MSE at a point $x$, and derive explicit expressions for the optimal tuning parameters $h$ and $a$.
Estimation of the derivatives of the generator is also discussed.
A simulation study shows the performance of the new methods. 
}}

\section{Introduction}
\label{sec:intro}

Elliptical distributions are a well-known family of distributions, generalizing spherically symmetric distributions (which are invariant by rotations).
In many situations, elliptical distributions allow easy computations such as the Value-at-Risk in finance, see e.g. \cite{Kamdem_VaRell}.
They also facilitate the inference and, in some sense, avoid the curse of dimensionality.
Indeed, a continuous $d$-dimensional elliptical distribution is parametrized by its mean $\mu \in \Rb^d$, its matrix $\Sigma$ assumed to be positive definite of dimension $d \times d$ and its density generator $g: \Rb_+ \to \Rb_+$.
We say that a random vector $\X$ follows the elliptical distribution $\Ec(\mubf, \Sigmabf, g)$ if its density is
\begin{align}\label{eq:elliptical density function}
    f_\X(\x) = {|\Sigmabf|}^{-1/2}
    g\left( (\x-\mubf)^\top \, \Sigmabf^{-1} \, (\x-\mubf) \right),
    \; \text{for any } \x \in \Rb^d.
\end{align}

\medskip

This construction reduces the complexity of the density estimation problem by reducing the estimation of a $d$-dimensional density function $f_\X$ to the estimation of the generator $g$, a $1$-dimensional function.

\medskip

Estimation of the generator $g$ has already been studied, first by \cite{stute1991nonparametric} who proposed a kernel-based estimator. This estimator was then improved by~\cite{liebscher2005_semiparametric}. An alternative estimator based on sieves was proposed by \cite{battey2014nonparametric}. Estimators of $g$ based on nonparametric maximum likelihood estimation via isotonic regression or splines were proposed by \cite{bhattacharyya2014adaptive} under the assumption that $g$ is monotone.
On a related area, elliptical distributions are the cornerstone of elliptical copula models, which also relies on a generator that needs to be estimated. Several properties of this model, including identifiability conditions, and an estimation algorithm have been studied in \cite{derumigny2022identifiability}. In a Bayesian framework, \cite{liang_NP_ElliptCop} proposes to estimate the generator of an elliptical copula using B-splines.

\medskip

In this article, we are interested in nonparametric estimation of $g$ in a general context, assuming only some smoothness conditions on $g$. For this, we will rely on Liebscher's kernel-based estimator~\cite{liebscher2005_semiparametric} defined for $\xi \in \Rb_+$ by
\begin{align}
    \widehat{g}_{n,h,a}(\xi)
    := \frac{
    \xi^{\frac{-d+2}{2}} 
    \psiaxiprime}{n h s_d}
    \sum_{i=1}^n \left[
    K\left( \frac{ \psiaxi - \psi_a(\xi_i) }{h} \right)
    + K\left( \frac{ \psiaxi + \psi_a(\xi_i) }{h} \right) \right],
    \label{eq:def:liebscher_estimator}
\end{align}
where
\begin{itemize}
    \item $\X_1, \dots, \X_n$ is an i.i.d. sample from the law $\Ec(\mubf, \Sigmabf, g)$,
    
    \item $\xi_i := (\X_i - \mubf)^\top \, \Sigmabf^{-1} \, (\X_i - \mubf)$,

    \item $h = h(n) > 0$ is the bandwidth and $K$ is a one-dimensional kernel with compact support
    such that $K(0) = 1$, $K$ is symmetric and $\int K = 1$.

    \item $a > 0$ is a second tuning parameter and
    $\psi_a(\xi)
    := -a + (a^{d/2} + \xi^{d/2})^{2/d}.$

    \item $s_d := \pi^{d/2} / \Gamma(d/2)$, where $\Gamma$ is the Gamma function.
\end{itemize}

\medskip

The density $f_{\X}$ can then be estimated by replacing $g$ in \eqref{eq:elliptical density function} with \eqref{eq:def:liebscher_estimator}, and denote the estimated $f_{\X}$ by $\widehat{f}_{n,h,a}$.
The estimator defined in \eqref{eq:def:liebscher_estimator} depends on two tuning parameters: the bandwidth $h$, which is a classical tuning parameter for kernel smoothing, and the ``extra'' tuning parameter $a$, which helps to reduce the bias at $0$.
In the case $a = 0$, i.e. $\psiaxi = \xi$, this estimator reduces to the estimator introduced in \cite{stute1991nonparametric}.

\medskip

Note that replacing $\Sigmabf$ by $c \Sigmabf$ and $g$ by $g(c^{-1} \, \cdot \,)$ for some $c \in \Rb$ gives the same distribution, see e.g. \cite[Theorem 2.6.2]{fang1990generalized}.
To make the problem identifiable, the identifiability constraint $\Sigmabf = \Var[\X]$ is usually assumed. This has the implicit assumption that all components of $\X$ have finite variance.
Actually, we do not need such an assumption here, as we assume that $\mubf$ and $\Sigmabf$ are known, which then uniquely identifies the density generator $g$. Possibilities of extension of our procedures to the case where $\mubf$ and $\Sigmabf$ are unknown are discussed in Section~\ref{sec:unknown_muSigma}.
In the rest of this paper, we focus on the simpler case where no nuisance parameter ($\mu$ and $\Sigma$) are to be estimated, only the (functional) parameter $g$.

\medskip

In order to use the estimator $\widehat{g}_{n,h,a}$, both tuning parameters $h$ and $a$ need to be chosen.
\cite{liebscher2005_semiparametric} shows that the estimator $\widehat{g}_{n,h,a}$ is uniformly convergent and asymptotically normal for any choice of $a$.
However, the finite-distance behavior strongly depends on the choice of $a$, and also on the choice of $h$, as usual in kernel smoothing.

\medskip

Therefore, in this article, we study the impact of the choice of $h$ and $a$ on the performance of $\widehat{g}_{n,h,a}$. In Section~\ref{sec:MSE}, we compute the MSE and discuss the optimal choice of $h$. Section~\ref{sec:derivatives} present a new method for the estimation of the derivatives of the generator. New methods for choosing $a$ are presented in Section~\ref{sec:choice_a}.
Finite distance properties of the estimator are presented in a simulation study in Section~\ref{sec:simulation}.

\medskip

The proposed methodology has been implemented and is available in the \texttt{R} package 
\texttt{ElliptCopulas}~\cite{derumigny2024elliptCopulas}, in the
functions \texttt{EllDistrEst.adapt} for the adaptive choice of $a$ and $h$, and \texttt{EllDistrDerivEst} for the estimation of the derivatives.

\bigskip

\noindent
\textbf{Notations.}
For a given generator $g$, and a given $a > 0$,
we define the functions $\widetilde{g}$, $z_a$, $\rho_a$ by
\begin{align}
    \tildegxi := \xi^{(d-2)/2} \gxi, \quad
    z_a(\xi) := \tildegxi / \psi_a'(\xi), \text{ and }
    \rho_a := z_a \circ \psi_a^{-1},
    \label{eq:def:rho_za_widetildeg}
\end{align}
for any $\xi > 0$.
These functions will be useful because the factor $\xi^{(d-2)/2}$ arises naturally when doing the change of variable to spherical coordinates
(see Lemma~\ref{lemma:change_variable_SW} and the discussion therein).
Note that $\rho_a$ is the density of $\psiaxi$, see \cite[p. 207]{liebscher2005_semiparametric}.

\section{MSE and optimal choice of \texorpdfstring{$h$}{h}}
\label{sec:MSE}

We start by giving a first-order expansion for the both the bias and the variance of the estimator $\widehat{g}_{n,h,a}$.
This theorem is proved in Section~\ref{proof:thm:MSE}.

\begin{thm}
    Let $h, a > 0$.
    Assume that $\rho_a$ is three times continuously differentiable in a neighborhood of $\psiaxi$ for some $\xi > 0$.
    Then
    \begin{equation*}
        \Bias_{n,h,a}(\xi)
        := \Eb\left[\widehat{g}_{n,h,a}(\xi)\right] - \gxi
        = h^2 \times \frac{\psiaxiprime \, 
        \rho_a^{\prime\prime} \left(\psiaxi \right)}
        {2 \, \xi^{(d-2)/2}} \int_{\Rb}K(w)w^2 \, dw
        + O(h^3),
    \end{equation*}
    and 
    \begin{equation*}
        \Var\left[ \widehat{g}_{n,h,a}(\xi) \right]
        = \frac{\psiaxiprime \gxi \norm{K}_2^2
        }{n h s_d \xi^{(d-2)/2}}
        + o ( (n h)^{-1} ),
    \end{equation*}
    as $h \to 0$ and $n \to \infty$.
\label{thm:MSE}
\end{thm}

We remark that the rates involved in the latter theorem do not depend on the dimension $d$. This confirms that such an elliptical assumption allows to avoid the curse of dimensionality.
This theorem is similar to classical results on univariate kernel smoothing, featuring $\norm{K}_2^2$ and $\int_{\Rb}K(w)w^2 \, dw$, as well as the rates $h^2$ for the bias and $1/(nh)$ for the variance, see e.g. \cite[Section 1.2]{tsybakov2009nonparametric}.

\medskip

The mean-square error is then 
\begin{align*}
    \MSE_{n,h,a}(\xi)
    := \Eb\left[ \big(\widehat{g}_{n,h,a}(\xi) - \gxi \big)^2 \right]
    = \AMSE_{n, h, a}
    + \, o( h^4 ) + o ( (n h)^{-1} ),
\end{align*}
where the asymptotic MSE is defined by
\begin{align*}
    \AMSE_{n, h, a}
    %
    &:= \frac{C_1}{nh} \psiaxiprime
    + C_2 \frac{h^4}{4} \Big[\psiaxiprime
    \cdot \rho_a''\left( \psiaxi \right) \Big]^2,
\end{align*}
for two positive constants $C_1$, $C_2$ defined by
\begin{align}
    &C_1 := \frac{\norm{K}_2^2}{s_d}
    \frac{\gxi}{\xi^{(d-2)/2}}, \qquad
    C_2 := \frac{1}{\xi^{d-2}}
    \left(\int_{\Rb}K(w)w^2 \, dw \right)^2. \nonumber
\end{align}
This asymptotic mean square error depends on the two tuning parameters $h$ and $a$. Since this expression has a simpler dependence on $h$, we start by computing the optimal $h$ that minimizes the AMSE for any given $a$.
\begin{prop}\label{prop:optimal AMSE_hopt}
    Let $n > 0$, $g$ be a generator and $a > 0$.
    Let $A := a^{d/2}$.
    Then $h \mapsto \AMSE_{n, h, a}$
    is minimized for the choice $\hopt$ defined by
    \begin{align}
        \hopt(n, a)
        &:= n^{-1/5} \left( \frac{C_1}{C_2} \right)^{1/5}
        \frac{\psiaxiprime}{\left| \widetilde{\AMSE}_{A} \right|^{2/5}},
        \label{eq:expression_hopt}
    \end{align}
    and
    \begin{align}
        \AMSE_{n, \hopt, a}
        &= n^{-4/5} \left( (C_1 C_2)^{4/5}
        + C_1^{4/5} C_2^{1/5} \right)
        \left| \widetilde{\AMSE}_{A} \right|^{2/5},
        \label{eq:AMSE_hopt}
    \end{align}
    where
    \begin{align}
        &\widetilde{\AMSE}_{A}
        := \frac{\rho_a'' \left( \psiaxi \right)}{\psiaxiprime}
        = \tildegxisecond
        + \frac{\widetilde{C}_3 A^2
        + \widetilde{C}_4 A
        + \widetilde{C}_5}{(\xi^{d/2} + A)^2},
    \end{align}
    and the constants are given by
    \begin{align*}
        \widetilde{C}_3
        &:= \frac{d-2}{4 \xi^2} \big(3 (d-2) \tildegxi
        - 6 \xi \tildegxiprime \big), \\
        \widetilde{C}_4
        &:= \frac{d-2}{4 \xi^2} \big(\tildegxi (d - 4)
        - 6 \tildegxiprime \xi^{(d+2)/2} \big), \\
        \widetilde{C}_5
        &:= \frac{d^2 - 4}{4 \xi^2} \tildegxi.
    \end{align*}
    \label{prop:optimal_h}
\end{prop}
This proposition is proved in Section~\ref{proof:prop:optimal_h}.
Note that, for the optimal choice of $h$, the $\AMSE$ decouples the role of $n$ and $a$.
Indeed, Equation~\eqref{eq:AMSE_hopt} decomposes the $\AMSE$ with the bandwidth $\hopt$ as a product of three terms:
the classical rate of convergence $n^{-4/5}$,
a constant term that only depends on $\gxi$,
and the term $\widetilde{\AMSE}_A$ that only depends on $a$
and on the generator $g$.
This means that any asymptotically optimal $\aopt$ only depends on the normalized generator $\tildegxi$, and not on the sample size $n$, unlike the asymptotically optimal bandwidth $\hopt$.
Furthermore, note that the asymptotically optimal bandwidth $\hopt$ can also be decomposed into three terms, in the same way as the $\AMSE$ itself.

\section{Estimation of the derivatives of the generator}
\label{sec:derivatives}


Followed by Proposition \ref{prop:optimal AMSE_hopt}, it appears that the AMSE depends on the derivatives of $\rho$. This means that we need to know the derivatives of the generator, since $\rho$ depends on $g$. In practice, we will not know the function $g$, and so its derivatives as well. Therefore, we decided to estimate the derivatives of $g$.  

\medskip

The construction of the estimators for the derivatives follow a similar method as \cite{engel1994iterative}, but adjusted for elliptical distributions. First, we fix two integers $k$ and $m$ such that $m \geq k + 2$ and $m - k$ is even.
We introduce the estimator
%
\begin{align*}
    \widehat{\eta}_{k,n,h,a}(\xi)
    = \frac{1}{n h^{k+1} s_d}
    \sum_{i=1}^n \left[
    K_k\left( \frac{ \psiaxi - \psi_a(\xi_i) }{h} \right)
    + K_k\left( \frac{ \psiaxi + \psi_a(\xi_i) }{h} \right) \right],
\end{align*}
where $K_k$ is a kernel function of order $(k,m)$,
i.e. $K_k$ is assumed to satisfy
\begin{equation*}
    \int K_k(x) x^j \, dx = 
    \begin{cases}
        0 \ & \ \text{if} \ j = 0, 1, \dotsc, k - 1, k + 1, \dotsc, m - 1,\\
        (-1)^k k! \ & \ \text{if} \ j = k,\\
        \mu_m(K_k) \neq 0 \ & \ \text{if} \ j = m.
    \end{cases}
\end{equation*}
As noted by \cite{engel1994iterative}, we can construct kernel $K_k$ of order $(k,m)$ by differentiating $k$ times a kernel function $K$ of order $m-k$.


\medskip

Surprisingly, for $k > 0$ the estimator $\widehat{\eta}_{k,n,h,a}(\xi)$ does not converge to a function of $\gxiderivk$. Instead, it consistently estimates the unknown quantity $\rho_a^{(k)}(\psiaxi)$.
This claim is stated in the following proposition and it is proved in Section \ref{sec: MSE estimating derivatives}.
\begin{prop}\label{prop: expec and variance eta hat}
    Assume that the kernel $K_k$ has a compact support and $\rho_a$ is $k$-times continuously differentiable in a neighborhood of $\psiaxi$ for some $\xi > 0$.
    Then we have
    \begin{align}
        &\Eb\left[ \widehat{\eta}_{k,n,h,a}(\xi) \right]
        = \rho_a^{(k)} \left( \psiaxi \right)
        + \mu_m(K_k) \frac{h^{m-k}}{m!} \rho_a^{(m)} \left(\psiaxi\right)
        + O(h^{m-k+1}).\label{eq: expectation eta hat} \\
        &\Var\left[ \widehat{\eta}_{k,n,h,a}(\xi) \right]
        \sim \frac{\gxi \norm{K_k}_2^2}{n h^{2k+1} s_d}.
    \end{align}
\end{prop}
However, it is possible to consistently estimate $g^{(k)}$ using $\widehat{\eta}_{0,n,h,a}, \dots, \widehat{\eta}_{k,n,h,a}$. This is a consequence of the following lemma.
\begin{lemma}
    For every $k \geq 0$, there exists $k+1$ functions 
    $\alpha_{0,k}(\xi), \dots, \alpha_{k,k}(\xi)$ such that
    \begin{align*}
        \eta_{k,a}(\xi) := 
        \rho_a^{(k)} \left( \psiaxi \right)
        = \sum_{i=0}^k \alpha_{i,k}(\xi) \gxiderivi,
    \end{align*}
    and the functions $\alpha_{i,k}$ do not depend on $g$, but only on $a$ and $d$.
    Furthermore, $\alpha_{k,k}(\xi) \neq 0$.
    \label{lemma:decomposition_rho_k_psi_gi}
\end{lemma}
This lemma is proved in Section \ref{proof:lemma:decomposition_rho_k_psi_gi}, where a general formula for $\alpha$ is given in Equation~\eqref{eq:def:expression_alpha}.
By Lemma~\ref{lemma:decomposition_rho_k_psi_gi}, we know that
$[\eta_{0,a}, \dots, \eta_{k,a}]^\top
= [\alpha_{i,j}]_{1 \leq i,j \leq k} \times
[g, \dots, g^{(k)}]^\top,$
where $\alpha_{i,j} := 0$ for $i > j$.
In this sense, estimating the derivatives of $g$ can be seen as a linear inverse problem with a triangular matrix.
This allows for a straightforward estimation of $\gderivk$ by inverting the matrix $[\alpha_{i,j}]$.
Indeed, one can write:
\begin{align*}
    \gderivk = \frac{1}{\alpha_{k,k}}
    \eta_{k,a}
    - \sum_{i=0}^{k-1} \frac{\alpha_{i,k}}{\alpha_{k,k}} \gderivi,
\end{align*}
from which a consistent estimator can be derived as
\begin{align*}
    \widehat{\gderivk} := \frac{1}{\alpha_{k,k}}
    \widehat{\eta}_{k,n,h,a}
    - \sum_{i=0}^{k-1} \frac{\alpha_{i,k}}{\alpha_{k,k}} \widehat{\gderivi}.
\end{align*}
As a particular case, for $k = 0$,
$\rho_a^{(0)} \left( \psiaxi \right)
= \xi^{(d-2)/2} \gxi / \psiaxiprime$,
hence $\alpha_{0,0}(\xi) = \xi^{(d-2)/2} / \psiaxiprime$.
We can recognize that the estimator $\widehat{g^{(0)}}$ is exactly Liebscher's estimator, which was indeed introduced for the estimation of $g = g^{(0)}$. 
For $k=1$, we obtain
\begin{align*}
    \rho_a' \left( \psiaxi \right)
    &= (\tildegxi / \psiaxiprime)' \times \frac{1}{\psiaxiprime} \\
    &= \frac{(d-2) \xi^{(d-4)/2} \psiaxiprime
    - 2 \xi^{(d-2)/2} \psiaxisecond}{2 \psiaxiprime^3} \times \gxi
    + \frac{\xi^{(d-2)/2}}{\psiaxiprime^2} \times \gxiprime.
\end{align*}

This means that a consistent estimator of $\gxiprime$ can be computed using the following estimator
\begin{align*}
    \widehat{g'}_{n,h,a}(\xi)
    &:= \left( \widehat{\eta}_{1,n,h,a}(\xi)
    - \frac{(d-2) \xi^{(d-4)/2} \psiaxiprime
    - 2 \xi^{(d-2)/2} \psiaxisecond}{\psiaxiprime^3}
    \times \widehat{g}_{n,h,a}(\xi)
    \right)\\
    &\times
    \frac{\psiaxiprime^2}{\xi^{(d-2)/2}}.
\end{align*}
For $k = 2$, we can use the formula above, or equivalently, Lemma~\ref{lemma:computation_g''},
which suggests to define the following consistent estimator of $\gxisecond$
as
\begin{align*}
    \widehat{g''}_{n,h,a}(\xi) &:= \frac{\widehat{\eta}_{2,n,h,a}(\xi) [\psiaxiprime]^3}{\xi^{(d-2)/2}} \\
    &+ \frac{ \psiaxisecond [
    \xi \widehat{g}_{n,h,a}(\xi)
    + 2 \xi \widehat{g'}_{n,h,a}(\xi)
    + (d-2) \widehat{g}_{n,h,a}(\xi) ]
    + \widehat{g}_{n,h,a}(\xi)
    \psiaxithird}{\xi \psiaxiprime}\\
    &+ \left(\frac{\psiaxisecond }{ \psiaxiprime } \right)^2 
    \left( \widehat{g'}_{n,h,a}(\xi) - 2 \widehat{g}_{n,h,a}(\xi)
    + \frac{d-2}{2} \frac{\widehat{g}_{n,h,a}(\xi)}{\xi} \right)\\
    &- (d-2) \frac{\widehat{g'}_{n,h,a}(\xi)}{\xi}
    - \frac{(d-2)(d-4)}{4} \frac{\widehat{g}_{n,h,a}(\xi)}{\xi^2}.
\end{align*}

\section{Data-driven choice of the tuning parameters}
\label{sec:choice_a}

\subsection{Estimation of \texorpdfstring{$\widetilde{\AMSE}$}{AMSEtilde}}

By Proposition~\ref{prop:optimal AMSE_hopt}, we know that 
$\widetilde{\AMSE}_{A}
\,  =  \, \rho_a'' ( \psiaxi) / \psiaxiprime$.
Note that $\psiaxiprime$ is a known function, and that $\rho_a'' ( \psiaxi)$ can be consistently estimated by 
$\widehat{\eta}_{k,n,h,a}(\xi)$.
This suggests to introduce an estimator of $\widetilde{\AMSE}$ defined by
\begin{align*}
    \widehat{\widetilde{\AMSE}}_{n,h,a}
    &:= \frac{\widehat{\eta}_{2,n,h,a}(\xi)}{\psiaxiprime} \\
    &= \frac{1
    }{n h s_d \psiaxiprime}
    \sum_{i=1}^n \left[
    K_2\left( \frac{ \psiaxi - \psi_a(\xi_i) }{h} \right)
    + K_2\left( \frac{ \psiaxi + \psi_a(\xi_i) }{h} \right) \right]
\end{align*}
where $K_2$ is a kernel of order $(2,2)$.
As a consequence of Proposition~\ref{prop: expec and variance eta hat}, we obtain the asymptotic expansions for the bias and the variance of $\widehat{\widetilde{\AMSE}}_{n,h,a}$.
\begin{cor}
    As $n \to \infty$, we have
    \begin{align*}
        \Eb\big[ \widehat{\widetilde{\AMSE}} \big]
        = \widetilde{\AMSE}_{A}
        + \mu_m(K_2) \frac{h^{2}}{4!}
        \frac{\rho_a^{(4)} \left(\psiaxi\right)}{\psiaxiprime}
        + O(h^{3}),
    \end{align*}
    and
    \begin{equation*}
        \Var\left[ \widehat{\widetilde{\AMSE}} \right]
        \sim \frac{\gxi \norm{K_k}_2^2}{n h^{2k+1} s_d \left(\psiaxiprime\right)^2}.
    \end{equation*}
\end{cor}
Remember that our goal is to find the value of $a$ (or equivalently, of $A$) that minimizes $\widetilde{\AMSE}$.
For this, we fix a first-step bandwidth, denoted by $h_1$, which gives a consistent estimator $\widehat{\widetilde{\AMSE}}_{n, h_1, a}(\xi)$, that can be used as proxy for the true quantity $\widetilde{\AMSE}$.
Therefore, a value of $a$ that minimizes the approximated criteria $\widehat{\widetilde{\AMSE}}$ will be close to the 
true minimizer $a_\textnormal{opt}$ of $\widetilde{\AMSE}$.
We denote by $\widehat{a}$ the minimizer of $a \mapsto \widehat{\widetilde{\AMSE}}_{n, h_1, a}(\xi)$ on a given grid $a_1, \dots, a_N$.
From this approximate minimizer, we can then compute the (approximate) optimal bandwidth using Equation~\eqref{eq:expression_hopt}.
This equation requires the knowledge of $\widetilde{\AMSE}$ (that we already estimated) and of $C_1$.

\medskip

$C_1$ can be estimated replacing $\gxi$ by the kernel-based estimator $\widehat{g}_{n,h_2,\widehat{a}}(\xi)$ for another first-step bandwidth $h_2$.
Note that we could use the same first-step bandwidth as for the estimation of $\widetilde{\AMSE}$, i.e. $h_1 = h_2$. 
Nevertheless, this may not be a good idea in practice because estimation of $\widetilde{\AMSE}$ depends on the properties of $\widehat{\eta}_{2,n,h,a}$. On the contrary, estimation of $\gxi$ depends on the properties of $\widehat{\eta}_{0,n,h,a}$ whose convergence rates (and therefore, optimal bandwidth) are different (see Proposition~\ref{prop: expec and variance eta hat}).
From an estimator 
\begin{align}
    \widehat{C_1} := \frac{\norm{K}_2^2}{s_d}
    \frac{\widehat{g}_{n,h_2,\widehat{a}}}{\xi^{(d-2)/2}}
    \label{eq:estimator_C1}
\end{align}
of $C_1$ and $\widehat{\widetilde{\AMSE}}_{n,h_1,\widehat{a}}$, we can compute an approximate optimal bandwidth by
\begin{align}
    \widehat{h}(n, \widehat{a})
    := n^{-1/5} \left( \frac{\widehat{C_1}}{C_2} \right)^{1/5}
    \frac{ \psi_{\widehat{a}}'(\xi)}{\left| 
    \widehat{\widetilde{\AMSE}}_{n,h_1,\widehat{a}}
    \right|^{2/5}}.
    \label{eq:approx_hopt}
\end{align}
Finally, we can compute Liebscher's estimator \eqref{eq:def:liebscher_estimator} of $\gxi$ using $\widehat{h}$ and $\widehat{a}$.
To sum up, we propose the following procedure detailed in Algorithm~\ref{algo:optimal_h_a}.

\begin{algorithm}[htbp]
\SetAlgoLined
    \vspace{0.1cm}
    \KwIn{Dataset, first-step bandwidths $h_1, h_2$, grid of possible values $a_1, \dots, a_N$}
    \For{$k \gets 1$ \KwTo{} $N$}{
        Compute the approximate criteria $\widehat{\widetilde{\AMSE}}_{n,h_1,a_k}$ \; }
    Compute $\widehat{a} := \arg\min_{k} \big| \widehat{\widetilde{\AMSE}}_{n,h_1,a_k} \big|$  \;
    Compute $\widehat{g}_{n,h_2,\widehat{a}}(\xi)$  \;
    Compute $\widehat{C_1}$ from $\widehat{g}_{n,h_2,\widehat{a}}(\xi)$ by Equation~\eqref{eq:estimator_C1}  \;
    Compute $\widehat{h}(n, \widehat{a})$ by Equation~\eqref{eq:approx_hopt}  \;
    \KwOut{An estimator $\widehat{g}_{n,\widehat{h},\widehat{a}}(\xi)$ using the (approximate) optimal tuning parameters }
\caption{Choice of optimal tuning parameters and estimation of $\gxi$}
\label{algo:optimal_h_a}
\end{algorithm}

\subsection{Direct optimization of \texorpdfstring{$\widetilde{\AMSE}$}{AMSEtilde}}

Since $\partial A / \partial a = (d/2) a^{d/2 - 1} > 0$,
we only need to investigate the critical point of 
$\widetilde{\AMSE}_A$ to obtain a critical point of
$\AMSE(\widehat{g}_{n,\hopt,a}(\x))$.
Note that it can also be minimized at $\widetilde{\AMSE}_A = 0$.
The following result is proved in Section~\ref{sec: proof minimizer tildeAMSE_A}.
\begin{prop}\label{prop:minimizer of tildeAMSE_A}
    Let $\Delta = \, \Delta(\xi) := \tildegxi / \tildegxiprime$ and
    \begin{align}\label{eq:expression of Astar}
        A^*
        = \dfrac{2 (d+2) \Delta
        - \Delta (d - 4) \xi^{(d-2)/2}
        + 6 \xi^{d}}
        {6 \xi^{d/2} (d-2) \Delta 
        - \Delta (d - 4)
        - 6 \xi^{(d+2)/2}}.
    \end{align}
    There exists a computable value $\Aopt(g) \in \{0, +\infty, A^*\}$ depending only on $\gxi$, $\gxiprime$ and $d$ such that
    $A \mapsto \widetilde{\AMSE}_{A}$ is minimized for $A$ tending to $\Aopt(g)$.
\end{prop}

Interestingly, the computation of the optimal tuning parameter $\aopt$ (or equivalently, $\Aopt$) only requires the knowledge of $\gxi$ and $\gxiprime$, while the computation of the optimal value $\hopt$ of the bandwidth also requires the knowledge of the second derivative $\gxisecond$.

%
%
%


\medskip

When we have chosen our value $a$, there remains to choose the tuning parameter $h$.
For this, we can rely on the optimal $h$ given by Equation~\eqref{eq:expression_hopt}
which is 
\begin{align*}
    \hopt(n, a)
    &:= n^{-1/5} \left( \frac{C_1}{C_2} \right)^{1/5}
    \frac{\psiaxiprime}{\widetilde{\AMSE}_{A}^{2/5}},
\end{align*}
However, this still does depend on the knowledge of $C_1$, which will not be available in practice.
Since $C_1 := \frac{\Gamma(d/2)}{\pi^{d/2}}
\xi^{(-d + 2) / 2} \gxi \norm{K}_2^2$,
it can be estimated by using $\hat g$ with a first-step bandwidth in order to get an estimate of $\gxi$.
This problem does not exist for $C_2$ which only depends on the choice of the kernel~$K$.

\section{Simulation study}
\label{sec:simulation}

In the previous sections, we have studied asymptotic properties of the estimators. As a complement to these theoretical results, we now empirically study the properties of the estimators for a finite sample size.
For this, we use a Monte-Carlo approach with $200$ replications, and run these simulations on a few nodes on the DelftBlue supercomputer \cite{DHPC2024}.
As a default family, we choose the Gaussian generator.
To set a simple setting, we choose $\mubf = 0$ and $\Sigmabf$ to be the identity matrix. Studying the influence of the correlations is left for future research. It is expected that, outside of the boundary cases (correlations close to $-1$ or $1$, or correlation matrix close to be not positive definite), the influence of the correlations should be limited.
This was shown through simulation studies in a related setting for the estimation of elliptical copula density, see \cite[Figures 8 and 9]{derumigny2022identifiability}.

\subsection{Influence of the tuning parameters on the MISE of Liebscher's estimator \texorpdfstring{$\widehat{g}_{n,h,a}$}{hat g\_ n,h,a} in dimension 3}

\begin{figure}[hptb]
    \centering
    \includegraphics[width = \textwidth]{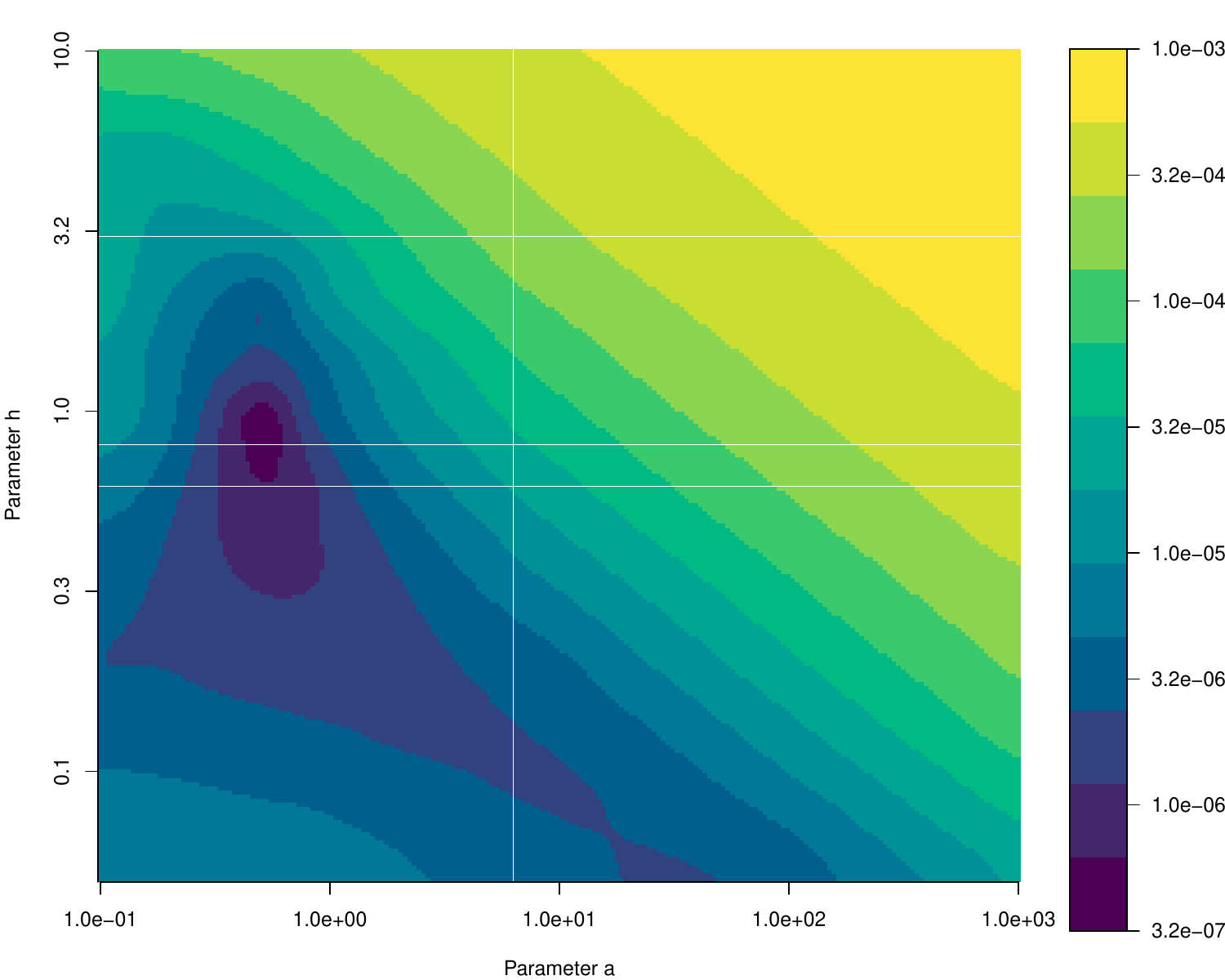}
    \caption{MISE of Liebscher's estimator $\widehat{g}_{n,h,a}$} as a function of the tuning parameters $a$ and $h$, for a sample size of $5000$ using the Gaussian generator. All axes are in log scale; labels are given in the form \texttt{1.0e+01} with the meaning $1.0 \times 10^{01}$.
    \label{fig:MISE_a_h_dim3}
\end{figure}

\begin{figure}[hptb]
    \centering
    \includegraphics[width = \textwidth]{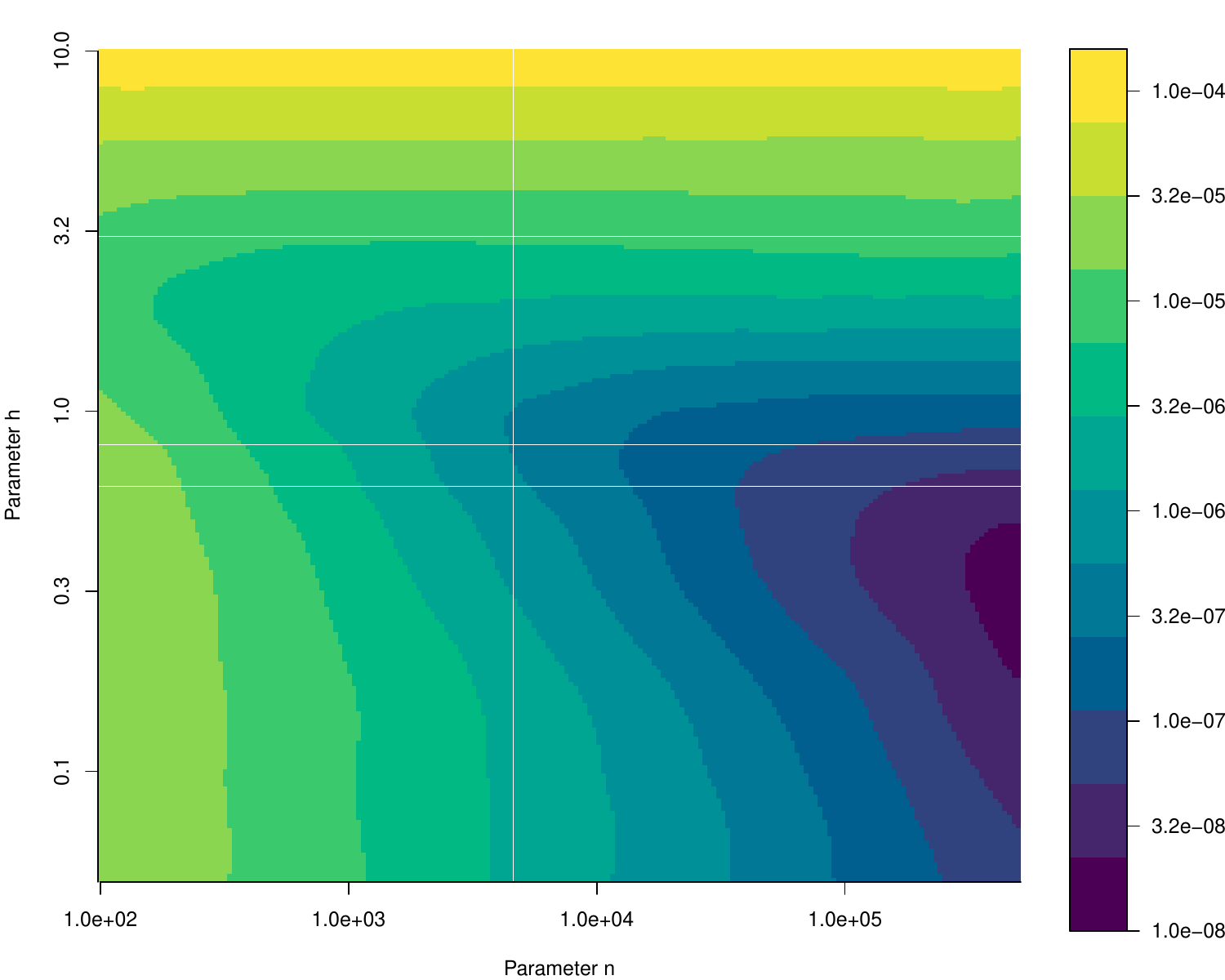}
    \caption{MISE of Liebscher's estimator $\widehat{g}_{n,h,a}$ as a function of the sample size $n$ and the tuning parameter $h$. 
    $a$ is chosen as the best tuning parameter for each pair $(n, h)$. All axes are in log scale; labels are given in the form \texttt{1.0e+01} with the meaning $1.0 \times 10^{01}$.}
    \label{fig:MISE_n_h_dim3}
\end{figure}

\begin{figure}[hptb]
    \centering
    \includegraphics[width = \textwidth]{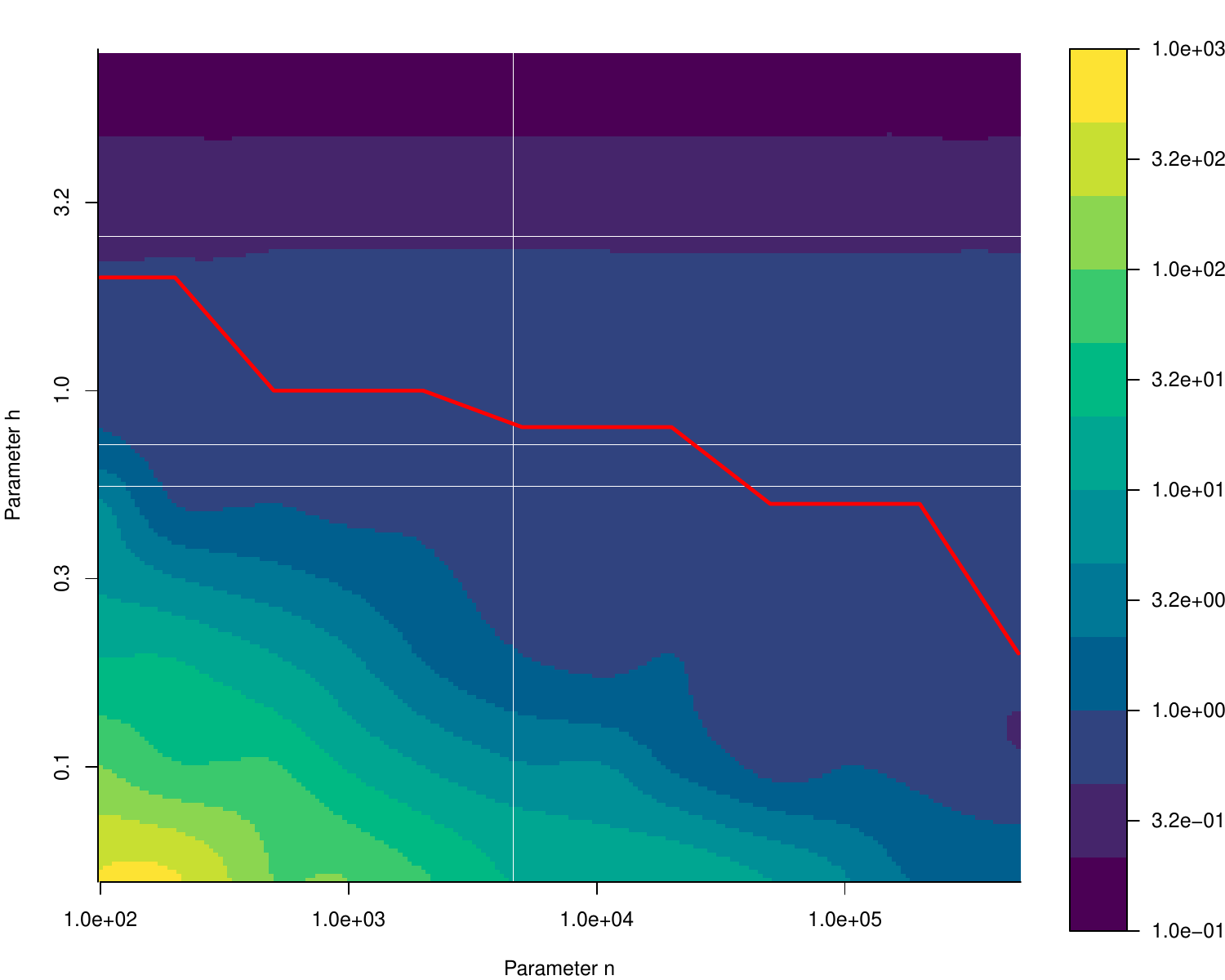}
    \caption{Heatmap of the optimal tuning parameter $\aopt$ as a function of the sample size $n$ and the tuning parameter $h$.
    The red curve represents the best $h$ as a function of $n$.
    All axes are in log scale; labels are given in the form \texttt{1.0e+01} with the meaning $1.0 \times 10^{01}$.}
    \label{fig:best_a_n_h_dim3}
\end{figure}

\begin{figure}[hptb]
    \centering
    \includegraphics[width = \textwidth]{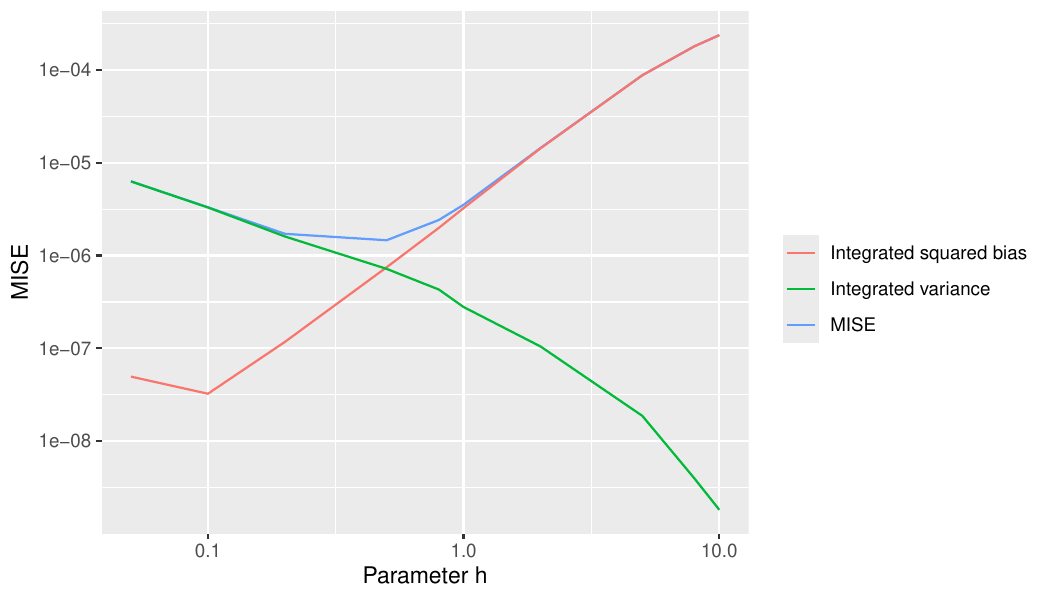}
    \caption{Decomposition of the MISE of Liebscher's estimator $\widehat{g}_{n,h,a}$ in terms of bias and variance, as a function of the tuning parameter $h$ for $n = 5000$ and $a = 1$.}
    \label{fig:BV_decomposition_MISE}
\end{figure}

\begin{figure}[hptb]
    \centering
    \includegraphics[width = \textwidth]{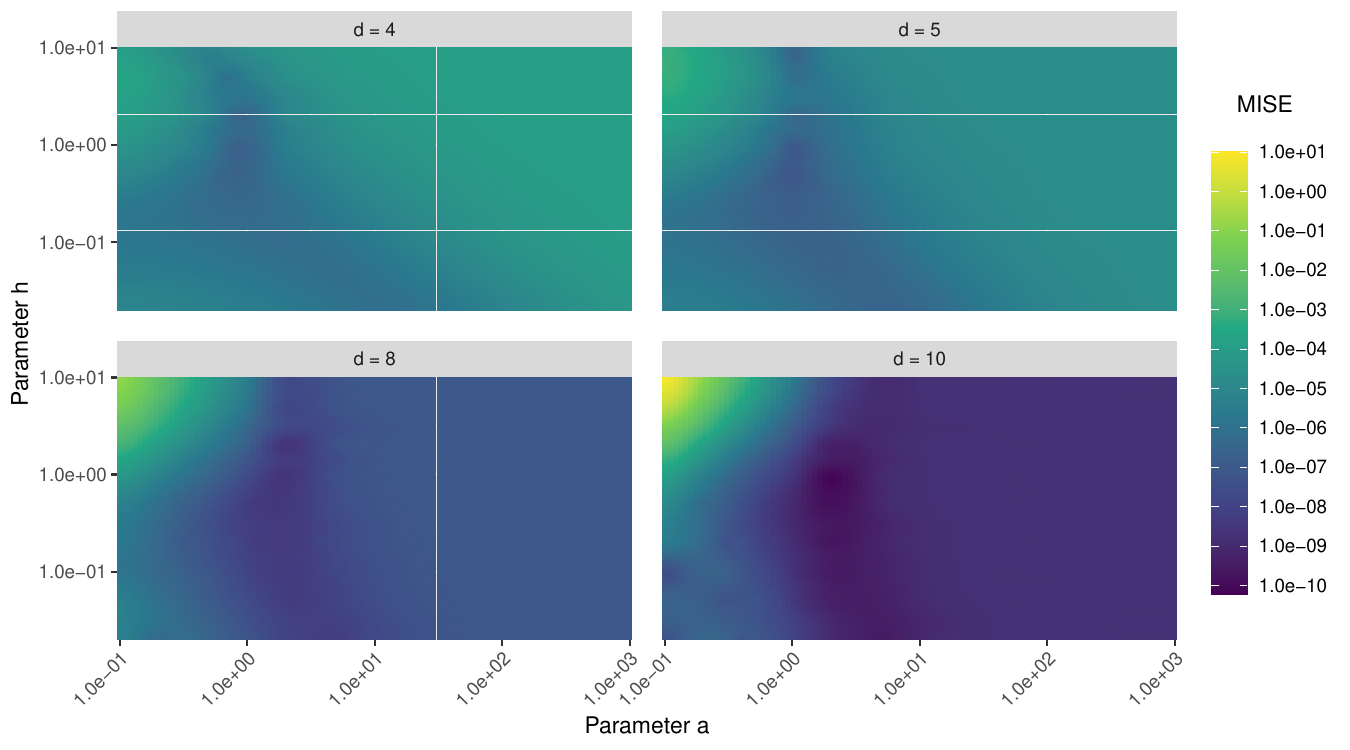}
    \caption{MISE of Liebscher's estimator $\widehat{g}_{n,h,a}$ as a function of the tuning parameters $a$ and $h$, for a sample size of $5000$ using the Gaussian generator and for different dimensions.}
    \label{fig:MISE_h_a_facet_d}
\end{figure}

\medskip

In the first simulation, we use a sample size $n = 5000$.
On this sample, we apply Liebscher's estimator~\eqref{eq:def:liebscher_estimator} with different values of the tuning parameters $a$ and $h$.
We study how these tuning parameters affect the mean integrated square error (MISE) defined by
\begin{align*}
    \MISE := \Expec \left[\int (\hat g(x) - g(x))^2 dx \right],
\end{align*}
for a given estimator $\hat g$ of $g$.
We approximate the integral by a discrete sum over the points $\{i / 10, i = 1, \dots, 50\}$.
The corresponding results are displayed in Figure~\ref{fig:MISE_a_h_dim3}.
There seems to be an optimal combination of $a$ and $h$ that minimizes the MISE.

\medskip

As expected, there exist values of $h$ and $a$ that minimizes the MISE. Note that, for each $h$ there is a best value $\aopt$ of $a$ that minimizes the MISE.
This optimal value of $a$ depends on $h$ and also on the sample size $n$ (in general).
To see the effect of the sample size $n$, we display the MISE as a function of $n$ and $h$, where $a$ is chosen as the optimal parameter.
Figure~\ref{fig:MISE_n_h_dim3} shows that the optimal tuning parameter $h$ decreases as a function of $n$; as expected the best MISE is also a decreasing function of the sample size.

\medskip

To find more about the behavior of $\aopt$, we have displayed its value as a function of $h$ and $n$ in Figure~\ref{fig:best_a_n_h_dim3}.
We overlay the curve of $(n, \hopt(n))$ to give insight about which values of $h$ are relevant.
Coherently with our theoretical findings,
the value of $\aopt(n, \hopt(n))$ seems roughly constant.
This means that there is a ``universal'' values of the tuning parameter $a$ which allows to reach the best MISE.

\medskip

The optimal value of $h$ is given by the balance between the bias term and the variance term 
(Figure~\ref{fig:BV_decomposition_MISE}).
This is the classical bias-variance trade-off in nonparametric statistics~\cite{derumigny2023lower}.
Empirically, it seems that the best bandwidth in dimension $3$ is $h(n) = 0.71 \times n^{0.21}$. This is obtain by fitting a linear regression of $\log_{10}(n)$ on $\log_{10}(h(n))$, where $h(n)$ is the empirically found best bandwidth $h$ for the sample size $n$.
However, since we only considered the example of the Gaussian generator, this can only be considered a rule-of-thumb and more research is needed to find under which conditions this rule-of-thumb is a good choice.

\bigskip

Similar phenomena can be observed for higher dimensions. We did the same simulations, but using dimensions
$d \in \{4, 5, 8, 10\}$.
The MISE as a function of the tuning parameters $a$ and $h$ is displayed for all four dimensions in Figure~\ref{fig:MISE_h_a_facet_d}.
For all of them, a value of the parameter $a$ around $1$ is close to optimal.

\FloatBarrier

\subsection{Comparison of the first-step estimator with the new estimator (adaptive choice of the tuning parameters)}

\bigskip

\begin{figure}[hptb]
    \centering
    \includegraphics[width = \textwidth]{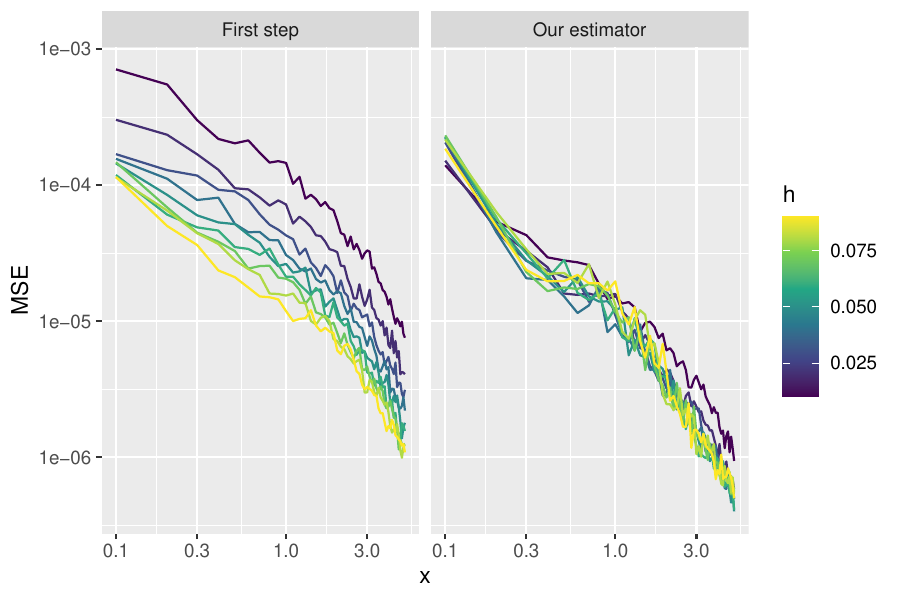}
    \caption{Comparison of the MSE of the first-step estimator and of the estimator given by Algorithm~\ref{algo:optimal_h_a} as a function of the tuning parameter $h$, for a sample size of $1000$ using the Gaussian generator.}
    \label{fig:MISE_comparison_estimators}
\end{figure}

\begin{figure}[hptb]
    \centering
    \includegraphics[width = \textwidth]{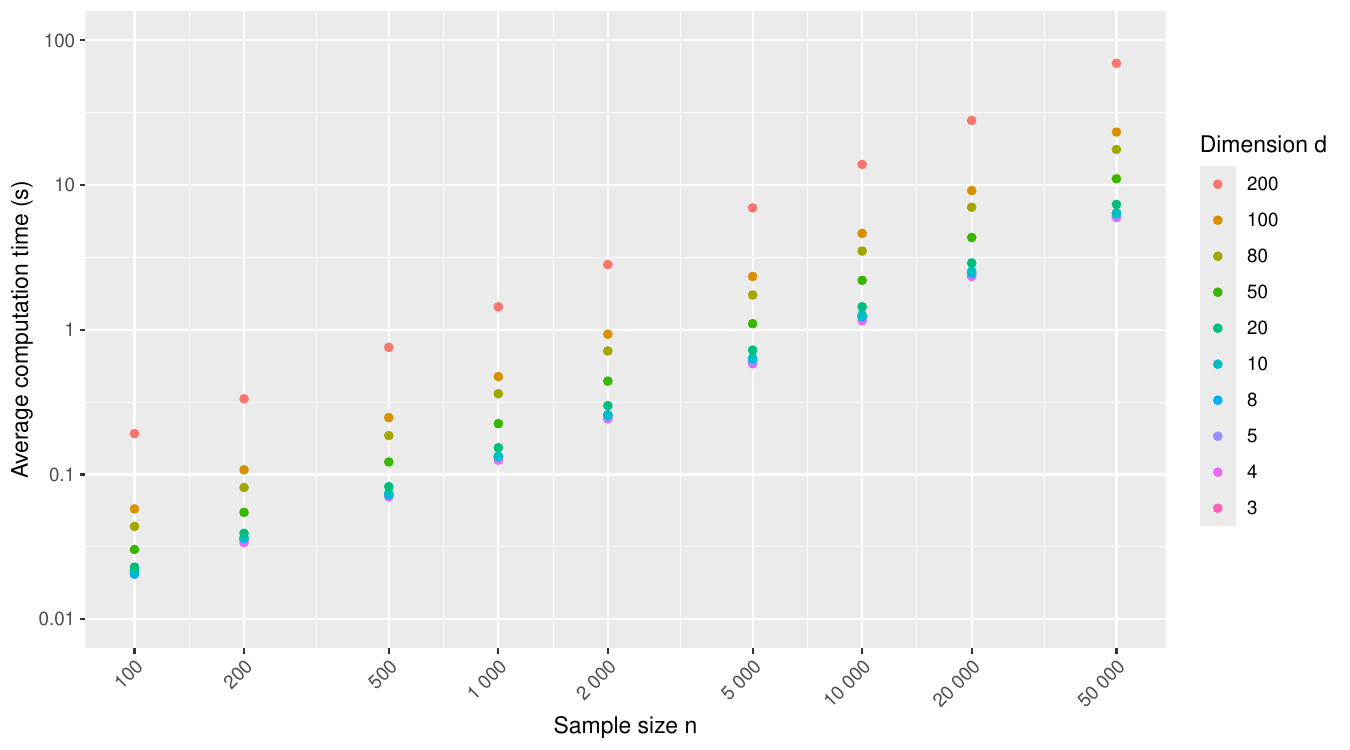}
    \caption{Average computation time (in seconds) of the estimator given by Algorithm~\ref{algo:optimal_h_a} as a function of the sample size $n$, for different dimensions.}
    \label{fig:plot_avgComputationTime}
\end{figure}

In general, the optimal values of the tuning parameters $\hopt$ and $\aopt$ of Liebscher's estimator depend on the true function $g$. 
In Algorithm~\ref{algo:optimal_h_a}, we propose a procedure to determine the best tuning parameters.
However, this procedure necessitates two first-step bandwidths $h_1$ and $h_2$.
In our simulation study and for simplicity, we choose these first-step bandwidths to be equal.
We compare the MSE of the first-step estimator and the MSE of the final estimator for different choices of $h$, and for different points $x$.

\medskip

These results are displayed in Figure~\ref{fig:MISE_comparison_estimators}.
We can see that the performance of the final estimator is quite insensitive to the choice of $h$, and allows to reach near-optimal performance without having to fine-tune the choice of $h$.
This highlights an advantage of our methodology compared to the existing one. In our new method, the choice of the (first-step) bandwidth becomes less critical to obtain good performances.

\medskip

The total computation of the simulations of the adaptive estimator was 693 hours. Depending on the sample size and the dimension, the computation time of the estimator given by Algorithm~\ref{algo:optimal_h_a} can vary from 0.02 seconds to 70 seconds.
Figure~\ref{fig:plot_avgComputationTime} illustrates the influence of the sample size and the dimension on the computation time, as expected.

\bigskip

\section{Extensions and open problems}

\subsection{Unknown \texorpdfstring{$\mubf$}{mu} and \texorpdfstring{$\Sigmabf$}{Sigma}}
\label{sec:unknown_muSigma}

For technical convenience, we have assumed before that $\mubf$ and $\Sigmabf$ were known.
In practice, this is not usually the case, and both needs to be estimated. $\mubf$ can be estimated via the sample mean (assuming that all the components of $\X$ have a finite first moment), and $\Sigmabf$ can be estimated via the sample covariance matrix (assuming that all the components of $\X$ have a finite second moment).
Robust estimators of $\mubf$ and $\Sigmabf$ are also available in the literature, see e.g. \cite{maronna1976robust}.

\medskip

If $\mubf$ and $\Sigmabf$ are replaced by estimators 
$\widehat\mubf = \widehat\mubf_{(\x_1, \dots, \x_n)}$ and 
$\widehat\Sigmabf = \widehat\Sigmabf_{(\x_1, \dots, \x_n)}$
in \eqref{eq:def:liebscher_estimator}, we obtain a corresponding estimator $\widecheck{g}_{n,h,a}$ given by
\begin{align*}
    &\frac{
    \xi^{\frac{-d+2}{2}} \psiaxiprime}{n h s_d}
    \sum_{i=1}^n
    K\left( \frac{ \psiaxi
    - \psi_a\left( \big( \x_i - \widehat\mubf_{(\x_1, \dots, \x_n)} \big)^\top \,
    \widehat\Sigmabf_{(\x_1, \dots, \x_n)}^{-1} \,
    \big( \x_i - \widehat\mubf_{(\x_1, \dots, \x_n)} \big) \right) }{h} \right)
    \\
    &+ K\left( \frac{ \psiaxi
    + \psi_a\left( \big( \x_i - \widehat\mubf_{(\x_1, \dots, \x_n)} \big)^\top \,
    \widehat\Sigmabf_{(\x_1, \dots, \x_n)}^{-1} \,
    \big( \x_i - \widehat\mubf_{(\x_1, \dots, \x_n)} \big) \right) }{h} \right),
\end{align*}
whose bias necessitates the computation of integrals of the form
\begin{align*}
    \int_{\Rb^{d(n-1)}} \int_{\Rb^d}
    & K\left(\frac{\psiaxi
    - \psi_a \left( \big( \x_i - \widehat\mubf_{(\x_1, \dots, \x_n)} \big)^\top \,
    \widehat\Sigmabf_{(\x_1, \dots, \x_n)}^{-1} \,
    \big( \x_i - \widehat\mubf_{(\x_1, \dots, \x_n)} \big) \right)}{h} \right) \\
    & \times {|\Sigmabf|}^{-1/2} g\big((\x_i - \mubf)^\top \,
    \Sigmabf^{-1} \, (\x_i - \mubf) \big) \, d\x_i
    \prod_{\scriptsize \begin{array}{c}
        1 \leq j \leq n \\
        j \neq i
    \end{array}} f_\X(\x_j) d\x_j.
\end{align*}
By the change of variable
$\z_i := \widehat\Sigmabf_{(\x_1, \dots, \x_n)}^{-1/2} (\x_i - \widehat\mubf_{(\x_1, \dots, \x_n)})$ and $\z_j = \x_j$ for $j \neq i$, the main factor in the integral can be rewritten as
\begin{align*}
    K\left(\frac{\psiaxi
    - \psi_a \big( \norm{\z_i}^2 \big) }{h} \right)
    {|\Sigmabf|}^{-1/2}
    g\big((\widehat\Sigmabf^{1/2} \z_i + \widehat\mubf - \mubf)^\top \,
    \Sigmabf^{-1} \,
    (\widehat\Sigmabf^{1/2} \z_i + \widehat\mubf - \mubf)
    \big)
    \, d\z_i \times J,
\end{align*}
where $J$ is the Jacobian of the transformation.
For usual estimators, this Jacobian could be computed, and if $g$ is Lipschitz, this would lead to a bound on the bias.
Such analysis is quite complicated and left for future research.

\medskip

Another possibility is to use a sample-splitting strategy (see e.g. the discussion at the end of Section 6 of \cite{derumigny2023testing} which considers sample splitting for hypothesis testing). Here, we would use a certain proportion $p \in (0, 1)$ of the observations is used for the estimation of $\mubf$ and $\Sigmabf$, and the rest of the data for the estimation of $g$. This could make the theoretical analysis easier (conditioning on the first sample so that $\mubf$ and $\Sigmabf$ are given, and then integrating with respect to them) but also introduces yet a third tuning parameter $p$ to be chosen. This is also not easy. A reasonable choice for $p$ is $50\%$; this was found to be nearly optimal in the simulation study done in \cite[Section 6 of the supplement]{derumigny2023testing}.

\subsection{Estimation of \texorpdfstring{$g$}{g} by minimization of the integrated MSE (MISE)}

In general, minimizing the mean integrated squared error may not be possible because $g$ itself needs not to be square-integrable.
To be an elliptical distribution generator, the only condition for $g$ to satisfy is
$\int t^{d/2 - 1} g(t) = 2 / s_d$,
where $s_d = 2 \pi^{d/2} / \Gamma(d/2)$
(see e.g. the introduction of \cite{derumigny2022identifiability}).
Therefore, if $g(t) = C / (t \sqrt{t})
\mathbf{1}_{t \leq 1}$ and $d = 4$,
then the above-mentioned condition can be satisfied since
$t^{4/2 - 1} g(t) = C / \sqrt{t}$
is integrable on $[0, 1]$.
On the contrary, $g^2(t) = 1/t^3$ is not integrable on $[0, 1]$.
This means that
\begin{align*}
    \MISE = \int_0^{+ \infty} (g - \widehat g_{n, h, a})^2
    = + \infty,
\end{align*}
since $\widehat g$ is always square-integrable (as a finite sum of compactly supported bounded functions).
Since the MISE is infinite for any choice of $h$, then this may not be the right tool for mathematical analysis.

\medskip

Under some restrictions on $g$ it may be possible to compute the MISE and to minimize it, but it would likely requires the use of other mathematical techniques and is left for future research.

\bigskip

\noindent
\textbf{Acknowledgments.} The authors thank the two anonymous reviewers for their useful comments which significantly improved the manuscript.

\bigskip

\bibliographystyle{abbrv}
\bibliography{main}{}


\FloatBarrier

\newpage

\noindent
{\LARGE Appendix}

\addcontentsline{toc}{section}{Appendix}

\setcounter{section}{0}
\renewcommand*{\theHsection}{Appendix \thesection}
\renewcommand{\thesection}{\Alph{section}}

\section{Proof of Theorem \ref{thm:MSE}}
\label{proof:thm:MSE}

\subsection{Expression for the bias}

%

\textbf{Step 1: Rewriting the bias in a simpler form.}
Because the random vectors $\X_1, \dots, \X_n$ are identically distributed, the expectation of $\widehat{g}_{n,h,a}$ can be written as
\begin{multline*}
    \Eb\left[\widehat{g}_{n,h,a}(\xi)\right]
    = \frac{\Gamma(d/2)}{\pi^{d/2}}
    \frac{\xi^{(-d + 2) / 2} \psiaxiprime}{h}
    \Bigg(
    \Eb\left[ K\left( \frac{ \psiaxi
    - \psiaXi{1}}{h} \right) \right] \\
    \qquad + \Eb\left[ K\left( \frac{ \psiaxi
    + \psiaXi{1}}{h} \right) \right] 
    \Bigg)
\end{multline*}
First, note that for $h$ small enough, 
\begin{align*}
    \Eb\left[ K\left( \frac{ \psiaxi + \psiaXi{1}}{h} \right) \right] = 0
\end{align*}
by Lemma \ref{lemma:integral_zero_large_n}.
Now, to compute the expectation of the other term,
we apply Lemma~\ref{lemma:change_variable_SW}, so that
\begin{multline*}
    \Eb\Bigg[K\Bigg(\frac{\psiaxi
    - \psiaXi{1}}{h} \Bigg)\Bigg]
    = \frac{\pi^{d/2}h}{\Gamma(d/2)} 
    \int_{-\psiaxi/h}^{\infty} K(w) \times \\
    \frac{g\left(\psi_a^{-1}\left(\psiaxi
    + wh\right)\right)\cdot
    \psi_a^{-1}\left(\psiaxi + wh\right)^{(d-2)/2}}{\psi_a'\left(\psi_a^{-1}\left(\psiaxi
    + wh\right)\right)}\ dw,
\end{multline*}
By defining now the function
\begin{equation} \label{eq:def:z_a}
    z(\xi) := \frac{g(\xi)}{\psi_a'(\xi)}\xi^{(d-2)/2},
\end{equation}
we get that
\begin{align*}
    \Bias_{h,a}(\xi)
    &= \xi^{(-d + 2) / 2} \psiaxiprime \int_{-\psiaxi/h}^{\infty}K(w)
    z \big(\psi_a^{-1}(\psiaxi + wh) \big) \ dw
    - \gxi \\
    &= \xi^{(-d + 2) / 2} \psiaxiprime I_1
    - \gxi,
\end{align*}
where 
\begin{align}
    I_1
    &:= \int_{-\psiaxi/h}^{\infty}
    K(w) z \big( \psi^{-1}_a 
    \left(\psiaxi + wh\right) \big) \, dw
    \nonumber \\
    &= \int_{-\infty}^{\infty}K(w) \cdot 
    z(\psi_a^{-1}\left(\psiaxi + wh\right))\ dw \nonumber \\
    &\qquad - \int_{-\infty}^{-\psiaxi/h} K(w)
    \cdot z(\psi_a^{-1}\left(\psiaxi + wh\right))\ dw 
    \label{eq:I_1}
\end{align}
We have $-\psiaxi / h \to - \infty$ as $n \to +\infty$.
Since we assume that $K$ has compact support, we deduce that the second term in \eqref{eq:I_1} is zero for $n$ large enough
by Lemma \ref{lemma:integral_zero_large_n}.
Note that 
\begin{align*}
    \gxi
    &= \frac{z(\xi)\psiaxiprime}{\xi^{(d-2)/2}}
    = \frac{\psiaxiprime}{\xi^{(d-2)/2}}
    \int_{-\infty}^{\infty} K(w)z(\xi)\ dw,
\end{align*}
where the last equality comes from the assumption that $\int_{\Rb} K(w)\ dw = 1$.
Therefore, we can rewrite the bias using the first term of \eqref{eq:I_1} and replacing $\gxi$ as above,
\begin{equation*}
    \Bias_{h,a}(\xi) = \xi^{(-d + 2) / 2} 
    \psiaxiprime \int_{-\infty}^{\infty}K(w)
    \left[z(\psi^{-1}_a\left(\psiaxi + wh\right)) - z(\xi) \right] \, dw. 
\end{equation*}
Remark that
$z(\psi^{-1}_a\left(\psiaxi + wh\right))
= (z \circ \psi^{-1}_a)\left(\psiaxi + wh\right)$
and $z(\xi) = (z \circ \psi^{-1}_a)\left(\psiaxi\right).$
Therefore, since $\rho_a = z \circ \psi^{-1}_a$ by Equation~\eqref{eq:def:rho_za_widetildeg}, we obtain
\begin{equation}
    \Bias_{h,a}(\xi) = \xi^{(-d + 2) / 2}
    \psiaxiprime \int_{-\infty}^{\infty} K(w)
    \left[ \rho_a\left(\psiaxi + wh\right) - \rho_a\left(\psiaxi\right) \right] \, dw.
    \label{prep}
\end{equation}

\medskip

\textbf{Step 2: Application of Taylor expansion.}
By applying Taylor-Lagrange formula to the function $\rho_a$ at the point 
$\psiaxi$, we get
\begin{align*}
    \rho_a\left(\psiaxi + wh\right)
    - \rho_a\left(\psiaxi\right)
    = \rho_a^{\prime}\left(\psiaxi\right)wh
    + \frac{\rho_a^{\prime\prime}
    \left(\psiaxi\right)}{2}(wh)^2
    +  \frac{\rho_a^{\prime\prime\prime}\left(c_T\right)}{6}(wh)^3
\end{align*}
where $c_T$ lies between $\psiaxi + wh$ and $\psiaxi$. So, \eqref{prep} becomes
\begin{align*}
    \Bias_{h,a}(\xi) &= \xi^{(-d + 2) / 2} \psi_a^{\prime}(\xi) \int_{-\infty}^{\infty} K(w) \Biggl[ 
    \rho_a^{\prime} \left(\psiaxi\right) wh\\
    &\hspace{12em} + \frac{\rho_a^{\prime\prime} \left(\psiaxi\right)}{2} (wh)^2
    + \frac{\rho_a^{\prime\prime\prime} \left(c_T\right)}{2}(wh)^3
    \Biggr] \, dw\\
    &= \frac{\psi_a^{\prime}(\xi)}{\xi^{(d - 2)/2}} \Biggl(\rho_a^{\prime}\left(\psiaxi\right) h \int_{\Rb} K(w)w \, dw \\
    &\hspace{5em} + h^2 \frac{\rho_a^{\prime\prime} \left(\psiaxi\right)}{2} \int_{\Rb}K(w)w^2 \, dw
     + h^3\int_{\Rb}K(w)\frac{\rho_a^{\prime\prime\prime} \left(c_T\right)}{6}w^3 dw \Biggr)
\end{align*}
Note that
$\int_{\Rb}K(w)w \, dw = 0$
since $K(w)$ is even and $w$ is odd; their product is odd and therefore the integral on $\Rb$ is zero.
As for 
$$h^3\int_{\Rb}K(w)\frac{\rho_a^{\prime\prime\prime}\left(c_T\right)}{6}w^3 dw,$$
the integral tends to a finite constant by Bochner's lemma. Indeed, using the assumption that $\rho_a^{\prime\prime\prime}$ is continuous, we obtain that $\rho_a^{\prime\prime\prime}$ is bounded on a neighborhood of $\psiax$.
We have now shown that
\begin{equation*}
    \Bias_{h,a}(\xi)
    = h^2 \times \frac{\psi_a^{\prime}(\xi)\rho_a^{\prime\prime}
    \left(\psiaxi\right)}{2\xi^{(d-2)/2}}
    \int_{\Rb}K(w)w^2 \, dw + O(h^3).
\end{equation*}

Note that we can replace $\rho_a''$ by its expression using $z$ and we get
\begin{equation*}
    \Bias_{h,a}(\xi)
    = h^2 \times \frac{
    z_a''\left( \xi \right) \cdot
    \psiaxiprime - \psiaxisecond \cdot \zaxiprime}{
    2\xi^{(d-2)/2} [\psiaxiprime]^2}
    \int_{\Rb}K(w)w^2 \, dw + O(h^3).
\end{equation*}

\subsection{Expression for the variance}

We know study the variance term.
By definition,
\begin{align*}
    \Var\left(\widehat{g}_{n,h,a}(\xi)\right) = \Eb\left[\widehat{g}_{n,h,a}(\xi)^2\right] - \Eb\left[\widehat{g}_{n,h,a}(\xi)\right]^2.
\end{align*}
The term $\Eb\left[\widehat{g}_{n,h,a}(\xi)\right]^2$ is already known, so we focus on
$\Eb\left[\widehat{g}_{n,h,a}(\x)^2\right]$.
Let 
$$c(\xi) := \frac{\Gamma(d/2) \xi^{-d/2 + 1} \psiaxiprime}{\pi^{d/2}}.$$
We have 
\begin{align*}
    \Eb \left[ \widehat{g}_{n,h,a}(\xi)^2 \right]
    = \frac{c^2(\xi)}{n^2 h(n)^2}
    \Eb\Biggl[\Biggl( \sum_{i=1}^n \Biggl[
    K \left( \frac{
    \psiaxi - \psiaXi{i}
    }{h(n)} \right)
    + K \left( \frac{
    \psiaxi + \psiaXi{i}
    }{h(n)} \right)
    \Biggr] \Biggr)^2 \Biggr]
\end{align*}
Expanding the sum, we find that:
\begin{multline*}
    \Eb \left[ \widehat{g}_{n,h,a}(\xi)^2 \right]
    = \frac{c^2(\xi)}{n^2 h(n)^2}
    \Bigg(
    \sum_{i=1}^n \left(T_{1,i} + T_{2,i} + 2 T_{3,i}\right)\\
    + \sum_{1 \leq i \neq j \leq n} 
    \left({T}_{4,i,j} + {T}_{5,i,j} +
    {T}_{6,i,j} + {T}_{7,i,j}\right)
    \Bigg),
\end{multline*}
where
\begin{align}
    {T}_{1,i}
    &:= \Eb\left[K^2\left(\frac{\psiaxi - \psiaXi{i}}{h(n)} \right) \right]\\
    {T}_{2,i}
    &:= \Eb\left[K^2\left(\frac{\psiaxi + \psiaXi{i}}{h(n)} \right) \right]\\
    {T}_{3,i}
    &:= \Eb\left[K\left(\frac{\psiaxi - \psiaXi{i}}{h(n)} \right)K\left(\frac{\psiaxi + \psiaXi{i}}{h(n)} \right)\right]\\
    {T}_{4,i,j}
    &:= \Eb\left[K\left(\frac{\psiaxi - \psiaXi{i}}{h(n)} \right)K\left(\frac{\psiaxi - \psiaXi{j}}{h(n)} \right) \right]\\
    {T}_{5,i,j}
    &:= \Eb\left[K\left(\frac{\psiaxi - \psiaXi{i}}{h(n)} \right)K\left(\frac{\psiaxi + \psiaXi{j}}{h(n)} \right) \right]\\
    {T}_{6,i,j}
    &:= \Eb\left[K\left(\frac{\psiaxi + \psiaXi{i}}{h(n)} \right)K\left(\frac{\psiaxi - \psiaXi{j}}{h(n)} \right) \right]\\
    {T}_{7,i,j}
    &:= \Eb\left[K\left(\frac{\psiaxi + \psiaXi{i}}{h(n)} \right)K\left(\frac{\psiaxi + \psiaXi{j}}{h(n)} \right) \right]
\end{align}
Therefore, because the random variables are identically distributed, we get
\begin{multline*}
    \Eb \left[ \widehat{g}_{n,h,a}(\xi)^2 \right]
    = \frac{c^2(\xi)}{n h(n)^2}
    \big( T_{1,1} + T_{2,1} + 2 T_{3,1}
    + (n - 1) {T}_{4,1,2}
    + (n - 1) {T}_{5,1,2}\\
    + (n - 1) {T}_{6,1,2}
    + (n - 1) {T}_{7,1,2}
    \big),
\end{multline*}

\textbf{Step 1. Removing terms that disappear.}
We start by applying Lemma \ref{lemma:integral_zero_large_n}.
This means that for all $n$ large enough,
\begin{align*}
    \Eb \left[ \widehat{g}_{n,h,a}(\xi)^2 \right]
    = \frac{c^2(\xi)}{n h(n)^2}
    \big( T_{1,1} + (n - 1) {T}_{4,1,2}
    \big).
\end{align*}

\textbf{Step 2. Computation of an equivalent of $T_{1,1}$.}
We appply Lemma~\ref{lemma:change_variable_SW}, with $K^2$ instead of $K$, and obtain
\begin{align*}
    {T}_{1,1} &= \frac{\pi^{d/2} h(n)}{\Gamma(d/2)}
    \int_{-\psiaxi / h}^{\infty} K^2(w)
    \frac{g\left(\psi^{-1}_a
    \left(\psiaxi + wh(n)\right)\right)
    \cdot \psi_a^{-1}\left(\psiaxi + wh\right)^{(d-2)/2} 
    }{\psi_a'\left(\psi^{-1}_a
    \left(\psiaxi + wh(n)\right)\right)} \ dw
\end{align*}
and so,
\begin{align*}
    \frac{c^2(\xi)}{nh(n)^2} T_{1,1}
    &= \frac{1}{nh(n)}
    \frac{\Gamma(d/2) \xi^{-d + 2}
    \psiaxiprime^2}{\pi^{d/2}} \\
    &\qquad \times \int_{w(0)}^{\infty} K^2(w)
    \frac{g\left(\psi_a^{-1}\left(\psiaxi + wh(n)\right)\right)
    \cdot \psi_a^{-1}\left(\psiaxi + wh\right)^{(d-2)/2} }
    {\psi_a'\left(\psi_a^{-1}\left(\psiaxi + w h(n)\right)\right)} \ dw \\
    &\sim \frac{1}{nh(n)} \frac{\Gamma(d/2)}{\pi^{d/2}}
    \xi^{(-d + 2) / 2} \psiaxiprime \gxi
    \int_{-\infty}^{\infty} K^2(w)\ dw,
\end{align*}
by Bochner's lemma.

\medskip

\textbf{Step 3. Computation of $T_{4,1,2}$.}
Since the $X_i, i=1, \dots, n$ are independent and identically distributed, we have
\begin{align*}
    {T}_{4,1,2}
    &=
    \Eb\left[K\left(\frac{\psiaxi
    - \psiaXi{1}}{h(n)} \right)\right] \cdot
    \Eb\left[K\left(\frac{\psiaxi
    - \psiaXi{2}}{h(n)} \right)\right]  \\
    &= \Eb\left[K\left(\frac{\psiaxi
    - \psiaXi{1}}{h(n)} \right)\right]^2  \\
    &= \Eb\left[K\left(\frac{\psiaxi
    - \psiaXi{1}}{h(n)} \right)\right]^2 \\
    &= \frac{h(n)^2}{c^2(\xi)} \Eb\left[\widehat{g}_{n,h,a}(\xi) \right]^2,
\end{align*}
for $n$ large enough by Lemma~\ref{lemma:integral_zero_large_n}.

\medskip

\textbf{Step 4. Combining all previous results.}
\begin{align*}
    \Var\left(\widehat{g}_{n,h,a}(\xi)\right)
    &= \Eb\left[\widehat{g}_{n,h,a}(\xi)^2\right]
    - \Eb\left[\widehat{g}_{n,h,a}(\xi)\right]^2 \\
    &= \frac{c^2(\xi)}{n h(n)^2}
    \big( T_{1,1} + (n - 1) {T}_{4,1,2} \big) 
    - \Eb\left[\widehat{g}_{n,h,a}(\xi)\right]^2 \\
    &= \frac{c^2(\xi)}{n h(n)^2}
    \big( {T}_{1,1} + (n - 1) {T}_{4,1,2}
    \big) - (g(\xi) - \Bias_{h,a}(\xi))^2 \\
    &= \frac{n-1}{n}(\gxi - \Bias_{h,a}(\xi))^2 + \frac{c^2(\xi)}{n h(n)^2}{T}_{1,1} - (\gxi - \Bias_{h,a}(\xi))^2 \\
    &= -\frac{1}{n}(\gxi - \Bias_{h,a}(\xi))^2  +
    \frac{c^2(\xi)}{n h(n)^2}{T}_{1,1}
\end{align*}
Since $\Bias_{h,a}(\xi) = o(1)$, we get
\begin{align*}
    \Var\left(\widehat{g}_{n,h,a}(\xi)\right)
    &= \frac{c^2(\xi)}{n h(n)^2}{T}_{1,1}
    + O(1/n) \\
    &= \frac{1}{nh(n)} \frac{\Gamma(d/2)}{\pi^{d/2}}
    \xi^{(-d + 2)/2} \psiaxiprime \gxi
    \norm{K}_2^2 + o(1/(n h(n)) ) + O(1/n) \\
\end{align*}
since $h(n) \to 0$ as $n \to \infty$.

\section{Proof of Proposition~\ref{prop:optimal_h}}
\label{proof:prop:optimal_h}

We start from the expression of the asymptotic mean square error
\begin{equation*}
    \AMSE_{n, h, a}
    := \frac{C_1}{nh} \psiaxiprime
    + C_2 \frac{h^4}{4} \psiaxiprime^2 \rho_a''^2\left( \psiaxi \right).
\end{equation*}
The partial derivative of $\AMSE_{n, h, a}$ with respect to $h$ is
\begin{align*}
    \frac{\partial}{\partial h} \AMSE_{n, h, a}
    = - \frac{C_1}{nh^2} \psiaxiprime 
    + C_2 h^3 \psiaxiprime^2 \rho_a''^2\left( \psiaxi \right).
\end{align*}
Solving $\frac{\partial}{\partial h}\AMSE_{n, h, a} = 0$ for $h$ gives the optimal bandwidth
\begin{align}
    \hopt(n) = \varphi_a(\xi) n^{-1/5},
    \label{eq:firstdef:hopt}
\end{align}
where
\begin{align}
    \varphi_a(\xi)
    := \left( \frac{C_1}{C_2 \psiaxiprime
    \rho_a''^2\left( \psiaxi \right)} \right)^{1/5}
    = C_3 \frac{\psi_a'}{(z''_a \psi_a' - \psi_a''z'_a)^{2/5}},
    \label{eq:def:varphi}
\end{align}
and $C_3 := (C_1 / C_2)^{1/5}$.
The second equality in Equation~\eqref{eq:def:varphi} is a consequence of Lemma~\ref{lemma:rho_second}, where the function $z_a$ is defined by $z_a(\xi) := \frac{g(\xi)}{\psi_a'(\xi)}\xi^{(d-2)/2}$.
Substituting $\hopt$ by its expression \eqref{eq:firstdef:hopt} in the AMSE gives
\begin{align*}
    \AMSE_{n,\hopt,a}
    &= \frac{C_1}{\varphi_a(\xi) n^{4/5}} \psiaxiprime
    + C_2 \varphi_a(\xi)^4 n^{-4/5} \psiaxiprime^2 \rho_a''^2 \left( \psiaxi \right) \\
    &= n^{-4/5}
    \bigg(C_1 \frac{\psi_a'}{\varphi_a}
    + C_2 \varphi_a^4 \psi_a'^2 (\rho_a''^2 \circ \psi_a) \bigg).
\end{align*}
By Equation~\eqref{eq:def:varphi}, we get that
$\psi_a' / \varphi_a
= C_3^{-1}(z''_a\psi_a' - \psi_a'' z'_a)^{2/5}$.
By Lemma~\ref{lemma:rho_second}, we also obtain
\begin{align*}
    \varphi_a^4 \psi_a'^2 (\rho_a''^2 \circ \psi_a)
    = C_3^4 \frac{\psi_a'^4}
    {(z''_a\psi_a' - \psi_a''z'_a)^{8/5}}
    \psi_a'^2 \left(
    \frac{z''_a\psi_a' - \psi_a''z'_a}{\psi_a'^3} \right)^2
    = C_3^4(z''_a\psi_a' - \psi_a'' z'_a)^{2/5}.
\end{align*}
Therefore,
\begin{align*}
    \AMSE_{n,\hopt,a}
    &= n^{-4/5} \left(\frac{C_1}{C_3} + C_2C_3^4 \right)
    \frac{\rho_a''\left( \psiaxi \right)^{2/5}}{\psiaxiprime^{2/5}} \nonumber \\
    &= n^{-4/5} \left(\frac{C_1}{C_3} + C_2C_3^4 \right)
    (z''_a\psi_a' - \psi_a'' z'_a)^{2/5} \nonumber \\
    &= n^{-4/5} \left( (C_1 C_2)^{4/5}
    + C_1^{4/5} C_2^{1/5} \right)
    (z''_a \psi_a' - \psi_a'' z'_a)^{2/5} \nonumber \\
    &= n^{-4/5} \left( (C_1 C_2)^{4/5}
    + C_1^{4/5} C_2^{1/5} \right)
    \left(\tildegxisecond
    + \frac{\widetilde{C}_3 A^2
    + \widetilde{C}_4 A
    + \widetilde{C}_5}{(\xi^{d/2} + A)^2} \right)^{2/5},
\end{align*}
by Lemma~\ref{lemma:z''psi'-psi''z'}.
The expression of $\hopt$ is obtained by combining Equations~\eqref{eq:firstdef:hopt}, \eqref{eq:def:varphi} and Lemma~\ref{lemma:z''psi'-psi''z'}.

\section{Proof of Proposition \ref{prop:minimizer of tildeAMSE_A}}
\label{sec: proof minimizer tildeAMSE_A}

In Proposition \ref{prop:optimal AMSE_hopt}, we show that the asymptotic MSE with optimal bandwidth $h$ has the following expression
\begin{equation*}
    \AMSE_{n,\hopt,a} = n^{-4/5} \left( (C_1 C_2)^{4/5}
    + C_1^{4/5} C_2^{1/5} \right)
    \left(\widetilde{g}''(\norm{\x}^2) 
    + \frac{\widetilde{C}_3 A^2
    + \widetilde{C}_4 A
    + \widetilde{C}_5}{(\xi^{d/2} + A)^2} \right)^{2/5},
\end{equation*}
where $A = a^{d/2}$. Therefore, we can minimize the function $a \mapsto \AMSE_{n,\hopt,a}$ by finding 
$$
\Aopt := \operatorname{argmin}_{A > 0} \left\{
    \frac{\widetilde{C}_3 A^2
    + \widetilde{C}_4 A
    + \widetilde{C}_5}{(\xi^{d/2} + A)^2}
    \right\},
$$
and so $\aopt := \Aopt^{2/d}$ minimizes $a \mapsto \AMSE_{n,\hopt,a}$. In this section, we analyze in which cases the minimum exists.

Followed from the definition of $\Aopt$, we define the following function. 
\begin{equation}\label{eq:definition of kappa}
    \kappa(A) := \frac{\widetilde{C}_3 A^2
    + \widetilde{C}_4 A
    + \widetilde{C}_5}{(\xi^{d/2} + A)^2}
\end{equation}
Note that $\kappa$ is a non-constant ratio of polynomials, and therefore it is an analytic function. Therefore, there exists a non-zero higher-order derivative of $\kappa$.
For the first derivative of $\kappa$, we have
\begin{align*}
    \kappa'(A) &= \dfrac{(2 \widetilde{C}_3 A + \widetilde{C}_4) (\xi^{d/2} + A)^2
    - (2 \xi^{d/2} + 2 A) (\widetilde{C}_3 A^2 + \widetilde{C}_4 A + \widetilde{C}_5) }{(\xi^{d/2} + A)^4} \\
    &= \dfrac{(2 \widetilde{C}_3 A + \widetilde{C}_4) (\xi^{d/2} + A)
    - 2 (\widetilde{C}_3 A^2 + \widetilde{C}_4 A + \widetilde{C}_5) }{(\xi^{d/2} + A)^3} \\
    &= \dfrac{\widetilde{C}_4 \xi^{d/2} + \widetilde{C}_4 A + 2 \widetilde{C}_3 \xi^{d/2} A
    - 2 \widetilde{C}_4 A - 2 \widetilde{C}_5}{(\xi^{d/2} + A)^3} \\
    &= \dfrac{A(2 \widetilde{C}_3 \xi^{d/2} - \widetilde{C}_4) + \widetilde{C}_4 \xi^{d/2} - 2 \widetilde{C}_5}{(\xi^{d/2} + A)^3}.
\end{align*}
We again differentiate $\kappa$ w.r.t $A$ to get 
\begin{align*}
    \kappa''(A) &= \dfrac{
    (2 \widetilde{C}_3 \xi^{d/2} - \widetilde{C}_4) (\xi^{d/2} + A)^3
    - 3 (\xi^{d/2} + A)^2 (A(2 \widetilde{C}_3 \xi^{d/2} - \widetilde{C}_4) + \widetilde{C}_4 \xi^{d/2} - 2 \widetilde{C}_5)
    }{(\xi^{d/2} + A)^6} \\
    &= \dfrac{
    (2 \widetilde{C}_3 \xi^{d/2} - \widetilde{C}_4) (\xi^{d/2} + A)
    - 3 (A(2 \widetilde{C}_3 \xi^{d/2} - \widetilde{C}_4) + \widetilde{C}_4 \xi^{d/2} - 2 \widetilde{C}_5)
    }{(\xi^{d/2} + A)^4} \\
    &= \dfrac{
    - 2 A (2 \widetilde{C}_3 \xi^{d/2} - \widetilde{C}_4)
    + 2 \widetilde{C}_3 \xi^d - 4 \widetilde{C}_4 \xi^{d/2} + 6 \widetilde{C}_5
    }{(\xi^{d/2} + A)^4}.
\end{align*}

From $\kappa'$, we determine for which $A$ is $\kappa$ maximum or minimum.
We have the following cases:
\begin{description}
    \item[Case 1.]{
    If $2 \widetilde{C}_3 \xi^{d/2} - \widetilde{C}_4 = 0$, then we consider the following three sub-cases.
    \begin{description}
        \item[Sub-case 1a.]{
        If $\widetilde{C}_4 \xi^{d/2} - 2 \widetilde{C}_5 < 0$, then $\kappa' < 0$. The function $\kappa$ decreases and its minimum is attained when $A \to + \infty$.
        }
        \item[Sub-case 1b.]{
        If $\widetilde{C}_4 \xi^{d/2} - 2 \widetilde{C}_5 = 0$, then $\kappa' \equiv 0$. This means that $\kappa$ is a constant function.
        }
        \item[Sub-case 1c.]{
        If $\widetilde{C}_4 \xi^{d/2} - 2 \widetilde{C}_5 > 0$, then $\kappa > 0$. Hence $\kappa$ reaches its minimum as $A \downarrow 0$, since $\kappa$ is an increasing function in this sub-case.
        }
    \end{description}
    }
    \item[Case 2.]{
    Suppose $2 \widetilde{C}_3 \xi^{d/2} - \widetilde{C}_4 \neq 0$. Let $A^* = (2 \widetilde{C}_5 - \widetilde{C}_4 \xi^{d/2}) / (2 \widetilde{C}_3 \xi^{d/2} - \widetilde{C}_4)$, then $\kappa'(A^*) = 0$. Therefore, 
    \begin{align*}
        \kappa''(A^*) &= \dfrac{
        - 2 (2 \widetilde{C}_5 - \widetilde{C}_4 \xi^{d/2})
        + 2 \widetilde{C}_3 \xi^d - 4 \widetilde{C}_4 \xi^{d/2} + 6 \widetilde{C}_5
        }{(C_6 + A^*)^4} \\
        &= \dfrac{2 \widetilde{C}_5 + 2 \widetilde{C}_3 \xi^d - 6 \widetilde{C}_4 \xi^{d/2}
        }{(\xi^{d/2} + A^*)^4}.
    \end{align*}
    Now we determine whether the (local) extrema of $\kappa$ is attained at $A^*$, we also have three sub-cases to consider.
    \begin{description}
        \item[Sub-case 2a.]{
        If $2 \widetilde{C}_5 + 2 \widetilde{C}_3 \xi^d - 6 \widetilde{C}_4 \xi^{d/2} > 0$, then $\kappa''(A^*) > 0$. There is only one extremum so the function $\kappa$ has a global minimum at $A^*$.
        }
        \item[Sub-case 2b.]{
        If $2 \widetilde{C}_5 + 2 \widetilde{C}_3 \xi^d - 6 \widetilde{C}_4 \xi^{d/2} < 0$, then $\kappa''(A^*) > 0$. This means that $A^*$ is a local maximum. However, there is only one extremum and so $A^*$ is a global maximum. The minimum is attained at either as $A \downarrow 0$ or $A \to \infty$. 
        }
        \item[Sub-case 2c.]{
        If $2 \widetilde{C}_5 + 2 \widetilde{C}_3 \xi^d - 6 \widetilde{C}_4 \xi^{d/2} = 0$, then the minimum could be attained at $0, +\infty$ or $A^*$.
        }
    \end{description}
    }
\end{description}
The expressions for $2 \widetilde{C}_3 \xi^{d/2} - \widetilde{C}_4$, $\widetilde{C}_4 \xi^{d/2} - 2 \widetilde{C}_5$ and $2 \widetilde{C}_5 + 2 \widetilde{C}_3 \xi^d - 6 \widetilde{C}_4 \xi^{d/2}$ are in Lemma \ref{lemma:expression for the conditions}.

From these cases, we conclude that $a \mapsto \AMSE_{n, \hopt, a}$ is minimized when $a \downarrow 0$, $a \to \infty$ or at $a = (A^*)^{2/d}$. The only case where we can compute $\aopt$ explicitly is in the sub-case 2a. Recall that 
$$
A^* = \frac{
2 \widetilde{C}_5 - \widetilde{C}_4 \xi^{d/2}
}{
2 \widetilde{C}_3 \xi^{d/2} - \widetilde{C}_4
}.
$$
Using the definitions of the constants given in Proposition \ref{prop:optimal AMSE_hopt}, we have 
\begin{align*}
    A^* &= \frac{
    2 \left[ \frac{d-2}{4 \xi^2} (d+2) \tildegxi \right]
    - \frac{d-2}{4 \xi^2}
    \big[\tildegxi (d - 4) - 6 \tildegxiprime \xi^{(d+2) / 2} \big] \xi^{d/2}
    }{
    2 \left[ \frac{d-2}{4 \xi^2} \big(3 (d-2) \tildegxi
    - 6 \tildegxiprime \xi \big) \right] \xi^{d/2}
    - \frac{d-2}{4 \xi^2}
    \big[\tildegxi (d - 4) - 6 \tildegxiprime \xi^{(d+2) / 2} \big]
    }\\
    &= \frac{
    2 \left[ (d+2) \tildegxi \right]
    - \big[ 
    \tildegxi (d - 4) - 6 \tildegxiprime \xi^{(d+2) / 2}
    \big] \xi^{d/2}
    }{
    2 \left[ \big(3 (d-2) \tildegxi
    - 6 \tildegxiprime \xi \big) \right] \xi^{d/2}
    - \big[
    \tildegxi (d - 4) - 6 \tildegxiprime \xi^{(d+2) / 2}
    \big]
    }\\
    &= \frac{
    \left[ 
    2 (d+2) - (d-4) \xi^{d/2} 
    \right] \tildegxi + 6 \xi^{d+1} \tildegxiprime
    }{
    \left[ 
    6 (d-2) \xi^{d/2} - d + 4
    \right] 
    \tildegxi - 6 \xi^{(d+2) / 2} \gxiprime 
    }.
\end{align*}
Let $\Delta(\xi) := \tildegxi/\tildegxiprime$, then 
\begin{align*}
    A^* = 
    \frac{
    \Delta(\xi) \left[ 2d + 4 - (d - 4) \xi^{d/2} \right] + 6 \xi^{d+1}
    }{
    \Delta(\xi) \left[ 6 \xi^{d/2} (d-2) - d + 4 \right] - 6 \xi^{(d+2) / 2},
    }
\end{align*}
which finishes the proof of Proposition \ref{prop:minimizer of tildeAMSE_A}.

\section{Technical results}

\subsection{Computation of integrals}

\begin{lemma}
    \label{lemma:integral_zero_large_n}
    Let $\psi$ be a measurable function from $\Rb^+$ to $\Rb^+$
    and $\xi > 0$ such that $\psixi > 0$.
    Assume that the kernel $K$ has compact support
    and $h(n) \to 0$ as $n \to \infty$.
    Let $Z$ be a bounded random variable.
    Let $I_n$ be defined by
    \begin{align*}
        I_n := \Expec \Bigg[ Z \cdot
        K\left(\frac{ \psi(\xi)
        + \psi \big((\X_i - \mubf)^\top \, \Sigmabf^{-1} \, (\X_i - \mubf) \big) }{h(n)} \right)
        \Bigg].
    \end{align*}
    Then there exists $N \in \Nb$ such that $I_n = 0$ for all $n \geq N$.
\end{lemma}

\medskip

\noindent
\textit{Proof:}
Let $[a,b]$ be the support of $K$, and remember
$\xi_i := (\X_i - \mubf)^\top \, \Sigmabf^{-1} \, (\X_i - \mubf)$.
We have
\begin{align*}
    I_n
    &:= \Expec \Bigg[ Z \cdot
    K\left(\frac{ \psi(\xi) + \psi(\xi_i) }{h(n)} \right)
    \Bigg]
    = \int z \cdot
    K\left(\frac{ \psixi + \psi(x) }{h(n)} \right)
    d\Prob(x, z),
\end{align*}
where $\Prob$ is the law of the random vector $(\xi_i, Z)$.
For any $x > 0$,
\begin{align*}
    \frac{ \psixi + \psi(x) }{h(n)}
    \geq \frac{ \psixi }{h(n)} \to +\infty.
\end{align*}
Therefore, for all $n$ large enough, we have
\begin{align*}
    \text{for all } x > 0, \,
    \frac{ \psixi + \psi(x) }{h(n)}
    > b,
\end{align*}
and therefore
\begin{align*}
    \text{for all } x > 0, \,
    K\left(\frac{ \psixi + \psi(x) }{h(n)} \right) = 0.
\end{align*}
Finally, we get $I_n = 0$ for $n$ large enough, as claimed.
$\Box$

\medskip

In many computations involving elliptical distributions, one can use a well-chosen change of variable to transform an expectation over the elliptically distributed variable $\X_1$ (i.e. a $d$-dimensional integral) into a one-dimensional integral.
This is formalized in the following Lemma, which is a consequence of the results given in~\cite{stute1991nonparametric}.
They show that by using $d$-dimensional spherical coordinates, we can simplify these integrals (see pages 174--175 and Lemma 4.3 of \cite{stute1991nonparametric}).

\begin{lemma}
    For any $\xi > 0$, for any $a, h > 0$, for any function $K: \Rb \to \Rb$, we have 
    \begin{multline*}
        \Eb\left[K\left(\frac{\psiaxi
        - \psiaXi{1}}{h} \right)\right]
        = \frac{\pi^{d/2}h}{\Gamma(d/2)}
        \int_{-\psiaxi/h}^{\infty} K(-w)\, \times \\
        \frac{g\left(\psi_a^{-1}\left(\psiax
        + wh\right)\right) \cdot
        \psi_a^{-1}\left(\psiaxi + wh\right)^{(d-2)/2}}{\psi_a'\left(\psi_a^{-1}\left(\psiax
        + wh\right)\right)}\ dw.
    \end{multline*}
\label{lemma:change_variable_SW}
\end{lemma}


\begin{proof}
%
%
We begin by doing the change of variables $\z_i := \Sigmabf^{-1/2} (\x_i - \mubf)$ and then we apply Stute and Werner's \cite{stute1991nonparametric} change of variable using the $d$-dimensional spherical coordinates, so that
\begin{align*}
    \Eb\left[K\left(\frac{\psiaxi
    - \psiaXi{1}}{h} \right)\right]
    &= \Eb\left[K\left(\frac{\psiaxi
    - \psi_a \big((\X_i - \mubf)^\top \, \Sigmabf^{-1} \, (\X_i - \mubf) \big)}{h} \right)\right] \\
    &= \int_{\Rb^d} K\left(\frac{\psiaxi
    - \psi_a \big(\norm{\z_i}^2 \big)}{h} \right)
    {|\Sigmabf|}^{-1/2} g\big(\norm{\z_i}^2 \big) \, {|\Sigmabf|}^{1/2} d\z_i \\
    &= \frac{\pi^{d/2}}{\Gamma(d/2)}
    \int_0^{\infty} K\left(\frac{\psiaxi
    - \psi_a\left(v_1\right)}{h} \right) g(v_1) v_1^{(d-2)/2} \, dv_1.
\end{align*}
We now do a third change of variables
$$w:= w(v_1) = \frac{-\psiaxi + \psi_a(v_1)}{h} \Longleftrightarrow v_1(w) = \psi_a^{-1}\left(\psiaxi + wh\right),$$
which gives
\begin{multline*}
    \Eb\left[K\left(\frac{\psiaxi
    - \psiaXi{i}}{h} \right)\right]
    = \frac{\pi^{d/2}h}{\Gamma(d/2)} \int_{w(0)}^{\infty} K(-w)\, \times \\ \frac{g\left(\psi_a^{-1}\left(\psiaxi
    + wh\right)\right)\cdot v_1(w)^{(d-2)/2}}{\psi_a'\left(\psi_a^{-1}\left(\psiaxi
    + wh\right)\right)}\ dw.
\end{multline*}
Noting that $w(0) = -\psiaxi/h$ finishes the proof. 
\end{proof}

\subsection{Technical computations}

We remind that $z_a$ is the function defined by
$z_a(\xi) := \dfrac{g(\xi)}{\psi_a'(\xi)}\xi^{(d-2)/2}$.

\begin{lemma}
    For any integer $d > 0$, and any $\xi, a > 0$, we have
    \begin{align*}
        \rho_a''\left( \psiaxi \right)
        =  \frac{\zaxisecond \cdot \psiaxiprime
        - \psiaxisecond \cdot \zaxiprime}{[\psiaxiprime]^3}.
    \end{align*}
\label{lemma:rho_second}
\end{lemma}


\begin{proof}
Let $\Psi_a(t)$ be the inverse function of $\psi_a$, and let $\Psi_a'$ be the derivative of $\Psi_a$ w.r.t. $t$.
Recall that $\rho_a = z_a \circ \Psi_a$. 
By differentiating this expression, we obtain $\rho_a' = (z'_a \circ \Psi_a)\cdot \Psi_a'$ and
\begin{align*}
    \rho_a'' &= (z'_a \circ \Psi_a)'\cdot \Psi_a'
    + (z'_a \circ \Psi_a) \cdot \Psi_a''
    = (z''_a \circ \Psi_a) \cdot \left(\Psi_a' \right)^2
    + (z'_a \circ \Psi_a) \cdot \Psi_a''.
\end{align*}
So, 
\begin{equation}\label{eq: rho double prime}
    \rho_a''\left( \psiaxi \right)
    = \zaxisecond \cdot \left(\Psi_a' \right)^2 \left(\psiaxi\right) 
    + \zaxiprime \cdot \Psi_a''\left(\psiaxi\right).
\end{equation}
Since,
\begin{align*}
    \Psi_a'(t) &= \frac{1}{\psi^{\prime}\left( \Psi_a(t) \right)}
    \ \ \ \text{and} \ \ \
    \Psi_a''(t) = -\frac{\psi_a''(\Psi_a(t))}{(\psi_a'(\Psi_a(t))^3},
\end{align*}
we get
\begin{equation}\label{eq: inverse derivatives}
    \left(\Psi_a' \right)^2\left(\psiaxi\right) = \frac{1}{\left[\psiaxiprime\right]^2} 
    \ \ \ \text{and} \ \ \ 
    \Psi_a''\left(\psiaxi\right)
    = -\frac{\psiaxisecond}{[\psiaxiprime]^3}.
\end{equation}
Therefore, combining Equations \eqref{eq: rho double prime} and \eqref{eq: inverse derivatives}, we get
\begin{align*}
    \rho_a''\left( \psiaxi \right)
    &= \frac{\zaxisecond}{\left[\psiaxiprime\right]^2} 
    - \frac{\psiaxisecond \cdot \zaxiprime}{[\psiaxiprime]^3}
    = \frac{\zaxisecond \cdot \psiaxiprime
    - \psiaxisecond \cdot \zaxiprime}{[\psiaxiprime]^3},
\end{align*}
as claimed.

\end{proof}

\begin{lemma}\label{lemma: z_a^prime and z_a^doubleprime}
    For any integer $d > 0$, and any $\xi, a > 0$, we have
    \begin{align}
    \zaxiprime
    &= \frac{\tildegxiprime \psiaxisecond + \tildegxi \psiaxiprime}{\left[ \psiaxiprime \right]^2}
    \label{eq:z_prime}, \\
    \zaxisecond
    &= \frac{2[\psiaxisecond]^2 \tildegxi
    - \psiaxiprime \psiaxithird \tildegxi 
    - 2 \tildegxiprime \psiaxiprime\psiaxisecond
    + \tildegxisecond [\psiaxiprime]^2}{\left[ \psiaxiprime \right]^3}.
    \label{eq:z_double_prime}
    \end{align}
    where 
    $$
    \tildegxi := \xi^{(d-2)/2} \gxi \qquad \text{and} \qquad \tildegxiprime := \xi^{(d-2)/2} \gxiprime
    + \frac{d-2}{2} \xi^{(d-4)/2} \gxi,
    $$
    Besides, we have that
    \begin{equation*}
    \tildegxisecond := (d-2) \gxiprime \xi^{(d-4)/2}
    + \gxisecond \xi^{(d-2)/2}
    + \frac{(d-2)(d-4)}{4} \gxi \xi^{(d-6)/2}.
    \end{equation*}
\end{lemma}

\noindent

\begin{proof}
By Equation~\eqref{eq:def:z_a}, $\zaxi := \gxi \xi^{(d-2)/2} / \psiaxiprime = \tildegxi / \psiaxiprime$ for $\xi > 0$. Differentiating this equality gives
\begin{align*}
    \zaxiprime
    = \frac{\tildegxiprime \psiaxiprime - \tildegxi \psiaxisecond}{\left[ \psiaxiprime \right]^2}.
\end{align*}
We differentiate this expression a second time to obtain
\begin{align*}
    \zaxisecond
    &= \frac{(\tildegxiprime \psiaxiprime - \tildegxi \psiaxisecond)'
    \left[ \psiaxiprime \right]^2
    - 2 \psiaxiprime \psiaxisecond
    (\tildegxiprime \psiaxiprime}{\left[ \psiaxiprime \right]^4}\\
    &\qquad - \frac{\tildegxi \psiaxisecond)}
    {\left[ \psiaxiprime \right]^4} \\
    &= \frac{(\tildegxiprime \psiaxiprime - \tildegxi \psiaxisecond)'
    \psiaxiprime
    - 2 \psiaxisecond (\tildegxiprime \psiaxiprime - \tildegxi \psiaxisecond)}
    {\left[ \psiaxiprime \right]^3} \\
    &= \frac{\big(\tildegxiprime \psiaxisecond + \tildegxisecond \psiaxiprime
    - \tildegxiprime \psiaxisecond - \tildegxi \psiaxithird \big)
    \psiaxiprime}{\left[ \psiaxiprime \right]^3}\\
    &\qquad - \frac{2 \psiaxisecond \big(\tildegxiprime \psiaxiprime - \tildegxi \psiaxisecond \big)}
    {\left[ \psiaxiprime \right]^3} \\
    &= \frac{\big(\tildegxisecond \psiaxiprime
    - \tildegxi \psiaxithird \big)
    \psiaxiprime
    - 2 \psiaxisecond \big(\tildegxiprime \psiaxiprime - \tildegxi \psiaxisecond \big)}
    {\left[ \psiaxiprime \right]^3} \\
    &= \frac{\tildegxisecond [\psiaxiprime]^2
    - \tildegxi \psiaxithird \psiaxiprime
    - 2 \tildegxiprime \psiaxisecond \psiaxiprime}{\left[ \psiaxiprime \right]^3}\\
    &\qquad + \frac{2 \tildegxi [\psiaxisecond]^2}
    {\left[ \psiaxiprime \right]^3},
    %
\end{align*}
as claimed.

\end{proof}

\begin{lemma}
    We have 
    \begin{align*}
        \gxisecond &= \frac{\rho_a''(\psiaxi) [\psiaxiprime]^3}{\xi^{(d-2)/2}}
        + \frac{ \psiaxisecond \times [ \xi \gxi + 2 \xi \gxiprime + (d-2) \gxi ] + \gxi \psiaxithird}{\xi \psiaxiprime}\\
        &\qquad
        + \left(\frac{\psiaxisecond }{ \psiaxiprime }\right)^2 
        \left( \gxiprime - 2 \gxi + \frac{d-2}{2} \frac{\gxi}{\xi} \right)
        \\
        &\qquad - (d-2) \frac{\gxiprime}{\xi}
        - \frac{(d-2)(d-4)}{4} \frac{\gxi}{\xi^2}
    \end{align*}
    \label{lemma:computation_g''}
\end{lemma}

\begin{proof}
    From Lemma \ref{lemma:rho_second} and Lemma \ref{lemma: z_a^prime and z_a^doubleprime}, we have
\begin{align*}
    \rho_a''(\psiaxi) &= \frac{2 \tildegxi [\psiaxisecond]^2 
    - \tildegxi \psiaxiprime \psiaxithird  
    - 2 \tildegxiprime \psiaxiprime\psiaxisecond
    }{[\psiaxiprime]^5}\\
    &\qquad
    + \frac{ \tildegxisecond [\psiaxiprime]^2 - \tildegxiprime [\psiaxisecond]^2 + \tildegxi \psiaxiprime \psiaxisecond }{[\psiaxiprime]^5},
\end{align*}
which yields 
\begin{align*}
    \tildegxisecond &= \frac{ \rho_a''(\psiaxi) [\psiaxiprime]^5
    + \tildegxiprime [\psiaxisecond]^2 + \tildegxi \psiaxiprime \psiaxisecond
    }{ [\psiaxiprime]^2 } \\
    &\qquad
    + \frac{+ \tildegxi \psiaxiprime \psiaxithird + 2 \tildegxiprime \psiaxiprime\psiaxisecond - 2 \tildegxi [\psiaxisecond]^2}{[\psiaxiprime]^2} \\
    &= \rho_a''(\psiaxi) [\psiaxiprime]^3 
    + \frac{\tildegxiprime [\psiaxisecond]^2}{[\psiaxiprime]^2} 
    + \frac{\tildegxi \psiaxisecond}{\psiaxiprime}\\
    &\qquad
    + \frac{\tildegxi \psiaxithird}{\psiaxiprime}
    + \frac{2 \tildegxiprime \psiaxisecond}{\psiaxiprime}
    -\frac{2 \tildegxi [\psiaxisecond]^2 }{ [\psiaxiprime]^2 } \\
    &= \rho_a''(\psiaxi) [\psiaxiprime]^3 
    + \frac{\tildegxi \psiaxisecond + \tildegxi \psiaxithird + 2 \tildegxiprime \psiaxisecond}{\psiaxiprime} \\
    &\qquad
    + \frac{\tildegxiprime [\psiaxisecond]^2 - 2 \tildegxi [\psiaxisecond]^2}{[\psiaxiprime]^2}.
\end{align*}
Using the definition of $\tildegxisecond$ given in Lemma \ref{lemma: z_a^prime and z_a^doubleprime}, we express the above expression in terms of $\gxisecond$.
It yields that 
\begin{align*}
    \gxisecond &= \rho_a''(\psiaxi) [\psiaxiprime]^3\\
    &\qquad 
    + \frac{\xi^{(d-2)/2} \gxi \psiaxisecond + \xi^{(d-2)/2} \gxi \psiaxithird
    + 2 \xi^{(d-2)/2} \gxiprime \psiaxisecond}{\psiaxiprime}\\
    &\qquad 
    + \frac{(d-2) \xi^{(d-4)/2} \gxi \psiaxisecond }{\psiaxiprime} \\
    &\qquad 
    + \frac{[ \xi^{(d-2)/2} \gxiprime + \frac{d-2}{2} \xi^{(d-4)/2} \gxi ] [\psiaxisecond]^2 - 2 \xi^{(d-2)/2} \gxi [\psiaxisecond]^2}{[\psiaxiprime]^2}\\
    &= \frac{1}{\xi^{(d-2)/2}} \times \Bigg( \rho_a''(\psiaxi) [\psiaxiprime]^3 \\
    &\qquad
    + \frac{\xi^{(d-4)/2} \psiaxisecond \times [ \xi \gxi + 2 \xi \gxiprime + (d-2) \gxi ] + \xi^{(d-2)/2} \gxi \psiaxithird}{\psiaxiprime}\\
    &\qquad 
    + \left(\frac{\psiaxisecond }{ \psiaxiprime }\right)^2 
    \left( \xi^{(d-2)/2}[ \gxiprime - 2 \gxi ] + \frac{d-2}{2} \xi^{(d-4)/2} \gxi \right)\\
    &\qquad
    - (d-2) \gxiprime \xi^{(d-4)/2} 
    - \frac{(d-2)(d-4)}{4} \gxi \xi^{(d-6)/2} \Bigg)\\
    &= \frac{\rho_a''(\psiaxi) [\psiaxiprime]^3}{\xi^{(d-2)/2}}
    + \frac{ \psiaxisecond \times [ \xi \gxi + 2 \xi \gxiprime + (d-2) \gxi ] + \gxi \psiaxithird}{\xi \psiaxiprime}\\
    &\qquad
    + \left(\frac{\psiaxisecond }{ \psiaxiprime }\right)^2 
    \left( \gxiprime - 2 \gxi + \frac{d-2}{2} \frac{\gxi}{\xi} \right)\\
    &\qquad 
    - (d-2) \frac{\gxiprime}{\xi}
    - \frac{(d-2)(d-4)}{4} \frac{\gxi}{\xi^2}
\end{align*}
\end{proof}

\begin{lemma}
    For any $\xi > 0$, we have
    \begin{align*}
        (z''_a\psi_a' - \psi_a'' z'_a)(\xi)
        = \tildegxisecond + \frac{\widetilde{C}_3 A^2
        + \widetilde{C}_4 A
        + \widetilde{C}_5}{\left(\xi^{d/2} + A\right)^2},
    \end{align*}
    where $A = A(a) := a^{d/2}$ and
    \begin{align}
        \widetilde{C}_3
        &:= \frac{d-2}{4 \xi^2} \big(3 (d-2) \tildegxi
        - 6 \tildegxiprime \xi \big)
        , \label{tilde C3} \\
        \widetilde{C}_4
        &:= \frac{d-2}{4 \xi^2} \big(\tildegxi (d - 4)
        - 6 \tildegxiprime\xi^{(d+2)/2} \big)
        , \label{tilde C4} \\
        \widetilde{C}_5
        &:= \frac{d^2 - 4}{4 \xi^2} \tildegxi
        , \label{tilde C5}
    \end{align}
\label{lemma:z''psi'-psi''z'}
\end{lemma}

\begin{proof}
We can express $z''_a\psi_a' - \psi_a'' z'_a$ in terms of derivatives of $\psi_a$, by using \eqref{eq:z_prime} and \eqref{eq:z_double_prime}.
We have 
\begin{align}
    z''_a\psi_a' - \psi_a'' z'_a
    &= \frac{2[\psiaxisecond]^2 \tildegxi
    - \psiaxiprime \psiaxithird \tildegxi 
    - 2 \tildegxiprime \psiaxiprime\psiaxisecond}{\left[ \psiaxiprime \right]^3} \psi_a'\nonumber\\
    &\qquad 
    + \frac{\tildegxisecond [\psiaxiprime]^2}{\left[ \psiaxiprime \right]^3} \psi_a' - \psi_a'' \frac{\tildegxiprime \psiaxiprime
    - \tildegxi \psiaxisecond}{\left[ \psiaxiprime \right]^2}
    \nonumber \\
    &= \frac{3 \tildegxi[ \psi''_a ]^2 - \tildegxi \psi_a'\psi_a'''
    - 3 \tildegxiprime \psi_a'\psi_a''
    + \tildegxisecond[ \psi'_a ]^2}{\left[ \psi'_a \right]^2}.
    \label{double tilde AMSE in terms of psi}
\end{align}

We now compute explicit expressions for the derivatives of $\psi_a$.
Remember that $\psi_a(\xi) = -a + (A + \xi^{d/2})^{2/d}$ and $\psiaxiprime = \xi^{\frac{d}{2} - 1} (A + \xi^{d/2})^{\frac{2}{d}-1}$.
Then, differentiating $\psiaxiprime$ yields
\begin{align*}
    \psiaxisecond &= \left( \frac{d}{2} - 1 \right) \xi^{\frac{d}{2}-2}
    \left( \xi^{\frac{d}{2}} + A \right)^{\frac{2}{d} - 1} +
    \left( \frac{2}{d} - 1 \right) \xi^{\frac{d}{2}-1}
    \left( \xi^{\frac{d}{2}} + A \right)^{\frac{2}{d} - 2} \frac{d}{2} \xi^{\frac{d}{2} - 1} \\
    &= \left( \frac{d}{2} - 1 \right) \xi^{\frac{d}{2}-2}
    \left( \xi^{\frac{d}{2}} + A \right)^{\frac{2}{d} - 1} +
    \left( 1 - \frac{d}{2} \right) \xi^{d-2}
    \left( \xi^{\frac{d}{2}} + A \right)^{\frac{2}{d} - 2}\\
    &= \frac{d-2}{2} \xi^{\frac{d}{2} - 2} \left( \xi^{\frac{d}{2}} + A \right)^{\frac{2}{d} - 2}
    \left( \left( \xi^{\frac{d}{2}} +  A \right)- \xi^{\frac{d}{2}} \right) \\
    &= \frac{d-2}{2} \xi^{\frac{d}{2} - 2} A \left( \xi^{\frac{d}{2}} + A \right)^{\frac{2}{d} - 2}.
\end{align*}
This can be recognized as 
\begin{align*}
    \psiaxisecond &= \psiaxiprime \times
    \frac{d-2}{2}\xi^{-1} A
    \left(\xi^{d/2}+ A \right)^{-1}.
\end{align*}
Differentiating again, we obtain
\begin{align*}
    \psiaxithird &= \left( \frac{d-2}{2} \right) \left( \frac{d}{2} - 2 \right) \xi^{\frac{d}{2} - 3} A
    \left( \xi^{\frac{d}{2}} + A \right)^{\frac{2}{d} - 2}\\
    &\qquad +
    \left( \frac{d-2}{2} \right) \left( \frac{2}{d} - 2 \right) \xi^{\frac{d}{2} - 2} A
    \left( \xi^{\frac{d}{2}} + A \right)^{\frac{2}{d} - 3} \frac{d}{2} \xi^{\frac{d}{2} - 1}\\
    &= \frac{(d-2)(d-4)}{4} \xi^{\frac{d}{2} - 3} A
    \left( \xi^{\frac{d}{2}} + A \right)^{\frac{2}{d} - 2} +
    \frac{(d-2)(1-d)}{2} \xi^{d-3} A \left( \xi^{\frac{d}{2}} + A \right)^{\frac{2}{d} - 3} \\
    &= \frac{d-2}{4} A \xi^{\frac{d}{2} - 3} \left( \xi^{\frac{d}{2}} + A \right)^{\frac{2}{d} - 3}
    \left(
    (d-4) \left( \xi^{\frac{d}{2}} + A \right) + 2(1-d) \xi^{\frac{d}{2}}
    \right) \\
    &= \frac{d-2}{4} A \xi^{\frac{d}{2} - 3} \left( \xi^{\frac{d}{2}} + A \right)^{\frac{2}{d} - 3}
    \left( 
    d\xi^{\frac{d}{2}} + dA - 
    4\xi^{\frac{d}{2}} - 4A + 2\xi^{\frac{d}{2}} - 2d\xi^{\frac{d}{2}}
    \right)\\
    &= \frac{d-2}{4} A \xi^{\frac{d}{2} - 3} \left( \xi^{\frac{d}{2}} + A \right)^{\frac{2}{d} - 3}
    \left( 
    - d\xi^{\frac{d}{2}}- 2\xi^{\frac{d}{2}} + dA - 4A 
    \right)\\
    &= \frac{d-2}{4} A \xi^{\frac{d}{2} - 3} \left( \xi^{\frac{d}{2}} + A \right)^{\frac{2}{d} - 3}
    \left( 
    (d-4)A - (d+2)\xi^{\frac{d}{2}}
    \right).
\end{align*}
This can also be recognized as 
\begin{align*}
    \psiaxithird &= \psiaxiprime \times
    \frac{d-2}{4} \xi^{-2} A
    \left( \xi^{d/2} + A \right)^{-2}
    \left( (d-4)A - (d+2)\xi^{\frac{d}{2}} \right)
\end{align*}
We have therefore shown that
\begin{align*}
    %
    \psiaxisecond
    &= \psiaxiprime \times
    \frac{d-2}{2}\xi^{-1} A
    \left(\xi^{d/2}+ A \right)^{-1},\\
    \psiaxithird
    &= \psiaxiprime \times
    \frac{d-2}{4} \xi^{-2} A
    \left( \xi^{d/2} + A \right)^{-2}
    \left(\left(d-4\right) A
    - \left(d+2\right) \xi^{d/2} \right).
\end{align*}
%
%
As a consequence, all terms in the numerator and the denominator of Equation~\eqref{double tilde AMSE in terms of psi} are proportional to $[\psi'_a]^2$, which can then be simplified, so that we obtain
\begin{align*}
    z''_a\psi_a' - \psi_a'' z'_a
    &= 3 \tildegxi \frac{[ \psi''_a ]^2}{\left[ \psi'_a \right]^2}
    - \tildegxi \frac{\psi_a'\psi_a'''}{\left[ \psi'_a \right]^2}
    - 3 \tildegxiprime \frac{\psi_a'\psi_a''}{\left[ \psi'_a \right]^2}
    + \tildegxisecond \\
    &= \tildegxisecond
    + \frac{d-2}{4 \xi^2 (\xi^{d/2} + A)^2} \times \\
    &\qquad \Big(
    3 (d-2) \tildegxi A^2
    - \tildegxi \big( (d - 4) A - (d + 2) \xi^{d/2} \big)  \\
    &
    \qquad 
    - 6 \tildegxiprime \xi A (\xi^{d/2} + A) \Big)  \\
    &= \tildegxisecond + \frac{\widetilde{C}_3 A^2
    + \widetilde{C}_4 A + \widetilde{C}_5}{\left(\xi^{d/2} + A\right)^2},
\end{align*}
finishing the proof.
\end{proof}

In the following lemma, we give an expression of $2 \widetilde{C}_3 \xi^{d/2} - \widetilde{C}_4$, $\widetilde{C}_4 \xi^{d/2} - 2 \widetilde{C}_5$ and $2 \widetilde{C}_5 + 2 \widetilde{C}_3 \xi^d - 6 \widetilde{C}_4 \xi^{d/2}$. These are conditions for the locations of the extrema of \eqref{eq:definition of kappa} in Section \ref{sec: proof minimizer tildeAMSE_A}. 

\begin{lemma}\label{lemma:expression for the conditions}
    Let $\xi > 0$ and $d > 2$. Let $\tilde{C}_3, \tilde{C}_4$ and $\tilde{C}_5$ be defined as in Proposition \ref{prop:optimal AMSE_hopt}. Then we have 
    \begin{align*}
        2 \widetilde{C}_3 \xi^{d/2} - \widetilde{C}_4 &=
        \frac{d-2}{4 \xi^2} \Big( 
        3 \xi^{d-1} (d-2) \gxi
        - \gxi (d - 4) \xi^{(d-2)/2}
        - 6 \gxiprime \xi^{d} \Big) \\
        2 \widetilde{C}_5 - \widetilde{C}_4 \xi^{d/2} &= 
        \frac{d-2}{4 \xi^2} \Big(
        2 (d+2) \xi^{(d-2)/2} \gxi
        - \xi^{(d-2)/2} \gxi (d - 4) \xi^{(d-2)/2}\\
        &\qquad 
       + 3 \xi^{(d-2)/2} \big(2 \gxiprime + (d-2) \xi^{-1} \gxi \big) \xi^{d} \Big)
    \end{align*}
    Lastly, 
    \begin{align*}
        2 \widetilde{C}_5 + 2 \widetilde{C}_3 \xi^d - 6 \widetilde{C}_4 \xi^{d/2} &= \frac{d-2}{4 \xi^2}
        \Bigg(
        2 (d+2) \xi^{(d-2)/2} \gxi
        + 6 (d-2) \xi^{(3d-2)/2} \gxi \\
        & \ \ \ \ \ \ -3\big(2 \gxiprime + (d-2) \xi^{-1} \gxi \big)
        \big(
        \xi^{(2d-2)/2}(d-4) - 4\xi^{3d/2}
        \big)
        \Bigg) 
    \end{align*}
\end{lemma}
\begin{proof}
We have:
\begin{align*}
    2 C_3 C_6 - C_4
    &= 2 \xi^{d/2} 
    \frac{d-2}{4 \xi^2} \Big(3 (d-2) \tildegxi
    - 6 \tildegxiprime \xi \Big)
    - \frac{d-2}{4 \xi^2} \Big(\tildegxi (d - 4)\\
    &\qquad
    - 6 \tildegxiprime\xi^{(d+2)/2} \Big) \\
    &= \frac{d-2}{4 \xi^2} \Big( 
    2 \xi^{d/2} \big(3 (d-2) \tildegxi - 6 \tildegxiprime \xi \big)
    - \big(\tildegxi (d - 4)\\
    &\qquad 
    - 6 \tildegxiprime\xi^{(d+2)/2} \big)
    \Big) \\
    &= \frac{d-2}{4 \xi^2} \Big( 
    6 \xi^{d/2} (d-2) \tildegxi 
    - \tildegxi (d - 4)
    - 6 \tildegxiprime\xi^{(d+2)/2} \Big)
\end{align*}
From Lemma~\ref{lemma:rho_second}, we get the expressions
$\tildegxi := \xi^{(d-2)/2} \gxi$,
and $2\tildegxiprime = \xi^{(d-2)/2}
\big(2 \gxiprime + (d-2) \xi^{-1} \gxi \big)$.
\begin{align*}
    2 C_3 C_6 - C_4
    &= \frac{d-2}{4 \xi^2} \xi^{(d-2)/2} \Big( 
    6 \xi^{d/2} (d-2) \gxi
    - \gxi (d - 4)\\
    &\qquad
   - 3 \big(2 \gxiprime + (d-2) \xi^{-1} \gxi \big)
    \xi^{(d+2)/2} \Big) \\
    &= \frac{d-2}{4 \xi^2} \xi^{(d-2)/2} \Big( 
    3 \xi^{d/2} (d-2) \gxi
    - \gxi (d - 4)
    - 6 \gxiprime \xi^{(d+2)/2} \Big) \\
    &= \frac{d-2}{4 \xi^2} \Big( 
    3 \xi^{d-1} (d-2) \gxi
    - \gxi (d - 4) \xi^{(d-2)/2}
    - 6 \gxiprime \xi^{d} \Big).
\end{align*}

Furthermore,
\begin{align*}
    2 C_5 - C_4 C_6
    &= 2 \frac{d-2}{4 \xi^2} (d+2) \tildegxi
    - \frac{d-2}{4 \xi^2} \big(\tildegxi (d - 4)
    - 6 \tildegxiprime\xi^{(d+2)/2} \big) \xi^{(d-2)/2} \\
    &= \frac{d-2}{4 \xi^2} \Big( 2 (d+2) \tildegxi
    - \big(\tildegxi (d - 4)
    - 6 \tildegxiprime\xi^{(d+2)/2} \big) \xi^{(d-2)/2} \Big) \\
    &= \frac{d-2}{4 \xi^2} \Big( 2 (d+2) \tildegxi
    - \tildegxi (d - 4) \xi^{(d-2)/2}
    + 6 \tildegxiprime \xi^{d} \Big).
\end{align*}
From Lemma~\ref{lemma:rho_second}, we get the expressions
$\tildegxi := \xi^{(d-2)/2} \gxi$,
and $2\tildegxiprime = \xi^{(d-2)/2}
\big(2 \gxiprime + (d-2) \xi^{-1} \gxi \big)$.
\begin{align*}
    2 C_5 - C_4 C_6
    &= \frac{d-2}{4 \xi^2} \Big(
    2 (d+2) \xi^{(d-2)/2} \gxi
    - \xi^{(d-2)/2} \gxi (d - 4) \xi^{(d-2)/2}\\
    &\qquad
    + 3 \xi^{(d-2)/2} \big(2 \gxiprime + (d-2) \xi^{-1} \gxi \big) \xi^{d} \Big) \\
\end{align*}

Using the fact that $ C_6 = \xi^{d/2}$, we get
\begin{align*}
    &2 C_5 + 2 C_3 C_6^2 - 6 C_4 C_6 \\
    &= 2 C_5 + 2 C_3 \xi^{d} - 6 C_4 \xi^{d/2} \\
    &= \frac{d-2}{4 \xi^2}
    \Bigg(
    2 (d+2) \tildegxi
    + 2 \xi^{d} \big(3 (d-2) \tildegxi
    - 6 \tildegxiprime \xi \big)
    - 6 \xi^{d/2} \big(\tildegxi (d - 4)\\
    &\qquad 
    - 6 \tildegxiprime\xi^{(d+2)/2} \big)
    \Bigg) \\
    &= \frac{d-2}{4 \xi^2}
    \Bigg(
    2 (d+2) \tildegxi
    + 6 \xi^{d} (d-2) \tildegxi
    - 12 \tildegxiprime \xi^{d+1}
    - 6 \xi^{d/2} \tildegxi (d - 4)\\
    &\qquad 
    + 36 \tildegxiprime\xi^{(2d+2)/2} 
    \Bigg) \\
    &= \frac{d-2}{4 \xi^2}
    \Bigg(
    2 (d+2) \xi^{(d-2)/2} \gxi
    + 6 (d-2) \xi^{(3d-2)/2} \gxi \\
    &\qquad -  6\xi^{(d-2)/2}
                  \big(2 \gxiprime + (d-2) \xi^{-1} \gxi \big) \xi^{d+1}\\
    &\qquad  - 3 \xi^{d/2} \xi^{(d-2)/2}
                  \big(2 \gxiprime + (d-2) \xi^{-1} \gxi \big) (d - 4) \\
    &\qquad + 18 \xi^{(d-2)/2}
                  \big(2 \gxiprime + (d-2) \xi^{-1} \gxi \big)\xi^{(2d+2)/2} 
    \Bigg)\\
    &= \frac{d-2}{4 \xi^2}
    \Bigg(
    2 (d+2) \xi^{(d-2)/2} \gxi
    + 6 (d-2) \xi^{(3d-2)/2} \gxi \\
    &\qquad -  6\xi^{3d/2}
    \big(2 \gxiprime + (d-2) \xi^{-1} \gxi \big) \\
    &\qquad - 3 \xi^{(2d-2)/2}
    \big(2 \gxiprime + (d-2) \xi^{-1} \gxi \big) (d - 4) \\
    &\qquad + 18 \xi^{3d/2}
    \big(2 \gxiprime + (d-2) \xi^{-1} \gxi \big)
    \Bigg)
\end{align*}
So,
\begin{align*}
    &2 C_5 + 2 C_3 C_6^2 - 6 C_4 C_6 \\
    &= \frac{d-2}{4 \xi^2}
    \Bigg(
    2 (d+2) \xi^{(d-2)/2} \gxi
    + 6 (d-2) \xi^{(3d-2)/2} \gxi \\
    &\qquad -  6\xi^{3d/2}
    \big(2 \gxiprime + (d-2) \xi^{-1} \gxi \big) \\
    &\qquad - 3 \xi^{(2d-2)/2}
    \big(2 \gxiprime + (d-2) \xi^{-1} \gxi \big) (d - 4) \\
    &\qquad + 18 \xi^{3d/2}
     \big(2 \gxiprime + (d-2) \xi^{-1} \gxi \big)
    \Bigg)\\
    &= \frac{d-2}{4 \xi^2}
    \Bigg(
    2 (d+2) \xi^{(d-2)/2} \gxi
    + 6 (d-2) \xi^{(3d-2)/2} \gxi \\
    &\qquad -3\big(2 \gxiprime + (d-2) \xi^{-1} \gxi \big)
    \big(
     2\xi^{3d/2} + \xi^{(2d-2)/2}(d-4) - 6\xi^{3d/2}
    \big)
    \Bigg)\\
    &= \frac{d-2}{4 \xi^2}
    \Bigg(
    2 (d+2) \xi^{(d-2)/2} \gxi
    + 6 (d-2) \xi^{(3d-2)/2} \gxi \\
    & \qquad -3\big(2 \gxiprime + (d-2) \xi^{-1} \gxi \big)
    \big(
     \xi^{(2d-2)/2}(d-4) - 4\xi^{3d/2}
    \big)
    \Bigg)
\end{align*}
\end{proof}

\subsubsection{Proof of Lemma~\ref{lemma:decomposition_rho_k_psi_gi}}
\label{proof:lemma:decomposition_rho_k_psi_gi}

\begin{proof}
    Let $\Psi = \psi^{-1}$, i.e. $\Psi(x)
    = \big((x+a)^{d/2} - a^{d/2} \big)^{2/d}$.
    By Faà di Bruno's formula, we have
    \begin{align*}
        \rho^{(k)}
        = (z \circ \Psi)^{(k)}
        = \sum_{m = 1}^k (z^{(m)} \circ \Psi) \times
        B_{k, m} \big( \Psi^{\prime}, 
        \Psi^{\prime\prime}, \dots,
        \Psi^{(k - m + 1)} \big),
    \end{align*}
    where $B_{m,k}$ are the partial exponential Bell polynomials (see e.g. \cite{dinardo2022kstatistics} and references therein).
    Therefore,
    \begin{align*}
        \rho^{(k)} \circ \psi
        = \sum_{m = 1}^k z^{(m)} \times
        \widetilde B_{k, m},
    \end{align*}
    where
    \begin{align*}
        \widetilde B_{k, m} := B_{k, m} \big( \Psi^{\prime}, 
        \Psi^{\prime\prime}, \dots,
        \Psi^{(k - m + 1))} \big) \circ \psi.
    \end{align*}
    Remark that $z(\xi) = \gxi \times \tau(\xi)$ where $\tau(\xi) := \xi^{(d-2)/2} / \psi'(\xi)$.
    Therefore, applying the general Leibniz rule, we obtain
    \begin{align*}
        \rho^{(k)} \circ \psi
        = \sum_{m = 1}^k (g \times \tau)^{(m)} \times
        \widetilde B_{k, m}
        = \sum_{m = 1}^k \sum_{i = 0}^{m}
        \binom{m}{i} g^{(i)} \tau^{(m - i)}
        \widetilde B_{k, m}
        = \sum_{i=0}^k \alpha_{i,k} g^{(i)},
    \end{align*}
    where
    \begin{align}
        &\alpha_{0, k} := \sum_{m = 1}^{k}
        \tau^{(m)} \widetilde B_{k, m},
        \,
        \text{ and for } i \geq 1, \,
        \alpha_{i, k} := \sum_{m = i}^{k} \binom{m}{i}
        \tau^{(m - i)} \widetilde B_{k, m}.
        \label{eq:def:expression_alpha}
    \end{align}
\end{proof}

Applying again Faà di Bruno's formula, it is possible to compute higher-order derivatives of $\Psi$ and $\tau$.
Indeed, $\Psi(x) = \big((x+a)^{d/2} - a^{d/2} \big)^{2/d} = f_1 \circ f_2$, where $f_1(x) = x^{2/d}$ and $f_2(x) = (x+a)^{d/2} - a^{d/2}$.
Therefore,
\begin{align*}
    \Psi^{(k)}
    = \sum_{m = 1}^k (f_1^{(m)} \circ f_2) \times
    B_{k, m} \big( f_2^{\prime}, f_2^{\prime\prime}, \dots,
    f_2^{(k - m + 1)} \big).
\end{align*}
Note that for $m \geq 1$,
\begin{align*}
    &f_1^{(m)}(x) = \frac{\prod_{i=0}^{m-1} (2 - i \times d)}{d^m}
    x^{(2/d) - m}, \text{ and } \\
    &f_2^{(m)}(x) = \frac{\prod_{i=0}^{m-1} (d - 2i)}{2^m} (x+a)^{(d-2m)/2}.
\end{align*}

We have $\tau(\xi) := \xi^{(d-2)/2} / \psi'(\xi)$.
Remember that $\psiaxiprime = \xi^{\frac{d}{2} - 1} (A + \xi^{d/2})^{\frac{2}{d}-1}$.
Therefore, $\tau(\xi) = (A + \xi^{d/2})^{(d-2)/d} = f_3 \circ f_4$,
where $f_3(x) = A + x^{d/2}$ and $f_4(x) = x^{(d-2)/d}$.
The computation of the derivatives of $\tau$ can be done in the same way as for $\Psi$.

\section{Proof of Proposition~\ref{prop: expec and variance eta hat}: MSE for Estimating Derivatives}
\label{sec: MSE estimating derivatives}

\subsection{Computation of the bias}

\begin{proof}
To simplify the notations, we write only $\widehat{\eta}$ instead of $\widehat{\eta}_{k,n,h,a}(\xi)$.
By linearity, we have
\begin{align*}
    \Expec\left[ \widehat{\eta} \right]
    = \frac{1}{ s_d h^{k+1}}
    \Bigg( &\Expec\left[ 
    K_k\left( \frac{ \psiaxi - \psi_a(\xi_1) }{h} \right)
    \right]
    + \Expec\left[
    K_k\left( \frac{ \psiaxi + \psi_a(\xi_1) }{h} \right)
    \right]\Bigg)
\end{align*}
By assumption, the kernel $K_k$ has a compact support; we can then apply Lemma \ref{lemma:integral_zero_large_n}, and obtain
\begin{equation*}
    K_k\left( \frac{ \psiaxi + \psi_a(\xi_1) }{h} \right) = 0,
\end{equation*}
for $n$ large enough.
Therefore, by applying Stute and Werner's \cite{stute1991nonparametric} change of variable
\begin{align*}
    \Expec\left[ \widehat{\eta} \right]
    &= \frac{1}{ s_d h^{k+1}} \int_{\Rb^d} K_k\left(\frac{\psiaxi
    - \psi_a\left(\xi_1\right)}{h} \right) g\left(\xi_1\right) |\Sigmabf|^{-1/2} \, d\xi_1 \\
    &= \frac{s_d}{s_d h^{k+1}} \int_0^{\infty} K_k\left(\frac{\psiaxi
    - \psi_a\left(v_1\right)}{h} \right) g(v_1) v_1^{(d-2)/2} \, dv_1.
\end{align*}
We now do a second change of variables
$$w:= w(v_1) = \frac{-\psiaxi + \psi_a(v_1)}{h} \Longleftrightarrow v_1(w) = \psi_a^{-1}\left(\psiaxi + wh\right),$$
which gives
\begin{equation*}
    \Expec\left[ \widehat{\eta} \right]
    = \frac{1}{h^{k+1}} h
    \int_{-\psiaxi/h}^{\infty} K_k(- w) \frac{g\left(\psi_a^{-1} \left( \psiaxi
    + wh\right)\right) \cdot v_1(w)^{(d-2)/2}}{\psi_a'\left(\psi_a^{-1}\left(\psiaxi
    + wh\right)\right)}\ dw,
\end{equation*}
We get that
\begin{align*}
    \Expec\left[ \widehat{\eta} \right]
    &= \frac{1}{h^k} \int_{-\psiaxi/h}^{\infty} K_k(-w)
    z(\psi_a^{-1}\left(\psiaxi + wh\right))\ dw \\
    &= \frac{1}{h^k} \int_{-\infty}^{\infty}K_k(-w) \cdot 
    z(\psi_a^{-1}\left(\psiaxi + wh\right))\ dw\\
    &\qquad 
    - \frac{1}{h^k} \int_{-\infty}^{-\psiaxi/h} K_k(-w)
    \cdot z(\psi_a^{-1}\left(\psiaxi + wh\right))\ dw 
\end{align*}
We have $(-\psiaxi) / h \to - \infty$ as $n \to +\infty$.
Since $K_k$ has compact support, we deduce that the second term in the above equation is zero for $n$ large enough
by Lemma \ref{lemma:integral_zero_large_n}.
We now then have
\begin{align*}
    \Expec\left[ \widehat{\eta} \right]
    &= \frac{1}{h^k} \int_{-\infty}^{\infty}K_k(-w) \cdot 
    z\left(\psi_a^{-1}\left(\psiaxi + wh\right)\right)\ dw \\
    &= \frac{1}{h^k} \int_{-\infty}^{\infty}K_k(-w) \cdot 
    \rho_a\left(\psiaxi + wh\right)\ dw,
\end{align*}
We now apply Taylor-Lagrange expansion at the order $m+1$ to the function $\rho_a$ at the point $\psiaxi$.
\begin{align*}
    \Expec\left[ \widehat{\eta} \right]
    &= \int_{-\infty}^{\infty}K_k(-w) \cdot 
    z\left(\psi_a^{-1}\left(\psiaxi + wh\right)\right)\ dw \\
    &= \sum_{i=0}^m
    \int_{-\infty}^{\infty}K_k(-w) \cdot w^i \frac{h^i}{i!}
    \rho_a^{(i)} \left(\psiaxi\right)\ dw\\
    &\qquad
    + \int_{-\infty}^{\infty}K_k(-w) \cdot w^{m+1} \frac{h^{m+1}}{(m+1)!}
    \rho_a^{(m+1)} \left(c_T\right)\ dw,
\end{align*}
where $c_T$ lies between $\psiaxi + wh$ and $\psiaxi$.
All the terms in the sum are zero except for $i=k$ and $i=m$.
Therefore,
\begin{align*}
    \Expec\left[ \widehat{\eta} \right]
    &= \int_{-\infty}^{\infty}K_k(-w) \cdot w^k \frac{h^k}{k!}
    \rho_a^{(k)} \left(\psiaxi\right)\ dw \\
    &\qquad + \int_{-\infty}^{\infty}K_k(-w) \cdot w^m \frac{h^m}{m!}
    \rho_a^{(m)} \left(\psiaxi\right)\ dw \\
    &\qquad + \int_{-\infty}^{\infty}K_k(-w) \cdot w^{m+1} \frac{h^{m+1}}{(m+1)!}
    \rho_a^{(m+1)} \left(c_T\right)\ dw,
\end{align*}
By the assumption on the kernel, we get
\begin{align*}
    \Expec\left[ \widehat{\eta} \right]
    &= \rho_a^{(k)} \left(\psiaxi\right)
    + \mu_m(K_k) \frac{h^{m-k}}{m!} \rho_a^{(m)} \left(\psiaxi\right)\\
    &\qquad + \frac{h^{m-k+1}}{(m+1)!}
    \int_{-\infty}^{\infty}K_k(-w) \cdot w^{m+1}
    \rho_a^{(m+1)} \left(c_T\right)\ dw .
\end{align*}
If $\rho_a^{(m+1)}$ is continuous, then $\rho_a^{(m+1)}$ is bounded around $\psiaxi$. Therefore, the last integral will converge to a finite constant by Bochner's lemma. So,
\begin{equation*}
    \Expec\left[ \widehat{\eta} \right] = \rho_a^{(k)} \left(\psiaxi\right)
    + \mu_m(K_k) \frac{h^{m-k}}{m!} \rho_a^{(m)} \left(\psiaxi\right) + O(h^{m-k+1}).
\end{equation*}
The bias of $\widehat{\eta}$ is 
$$
\Bias_{\widehat{\eta}}(\xi) = \mu_m(K_k) \frac{h^{m-k}}{m!} \rho_a^{(m)} \left(\psiaxi\right) + O(h^{m-k+1}).
$$

\end{proof}

\subsection{Variance for estimating derivatives and the corresponding MSE}

We now study the variance term.
From the definition of the variance, we have
$$
\Var\left( \widehat{\eta}\right) = \Eb\left[ \widehat{\eta}^2 \right] - \Eb\left[ \widehat{\eta} \right]^2.
$$
Since that the term
$\Eb\left[ \widehat{\eta} \right]^2$ is already known, we only need to study the term
$\Eb\left[ \widehat{\eta}^2 \right]$.
We have
\begin{align*}
    \Eb \left[ \widehat{\eta}^2 \right]
    = \frac{\Gamma(d/2)^2}{\pi^d n^2 h^{2(k+1)}}
    \Eb\Biggl[\Biggl( \sum_{i=1}^n \Biggl[
    K_k \left( \frac{
    \psiaxi - \psiaXi{i}
    }{h} \right)
    + K_k \left( \frac{
    \psiaxi + \psiaXi{i}
    }{h} \right)
    \Biggr] \Biggr)^2 \Biggr]
\end{align*}
Expanding the sum, we find that
\begin{align*}
    \Eb \left[ \widehat{\eta}^2 \right]
    = \frac{\Gamma(d/2)^2}{\pi^d n^2 h^{2(k+1)}}
    \Bigg(
    \sum_{i=1}^n \left(T_{1,i} + T_{2,i} + 2 T_{3,i}\right)
    + \sum_{1 \leq i \neq j \leq n} 
    \left({T}_{4,i,j} + {T}_{5,i,j} +
    {T}_{6,i,j} + {T}_{7,i,j}\right)
    \Bigg),
\end{align*}
where:
\begin{align*}
    {T}_{1,i} &:= \Eb\left[K_k^2\left(\frac{\psiaxi - \psiaXi{i}}{h} \right) \right]\\
    {T}_{2,i} &:= \Eb\left[K_k^2\left(\frac{\psiaxi + \psiaXi{i}}{h} \right) \right]\\
    {T}_{3,i} &:= \Eb\left[K_k\left(\frac{\psiaxi - \psiaXi{i}}{h} \right)K_k\left(\frac{\psiaxi + \psiaXi{i}}{h} \right)\right]\\
    {T}_{4,i,j} &:=
    \Eb\left[K_k\left(\frac{\psiaxi - \psiaXi{i}}{h} \right)K_k\left(\frac{\psiaxi - \psiaXi{j}}{h} \right) \right]\\
    {T}_{5,i,j} &:=
    \Eb\left[K_k\left(\frac{\psiaxi - \psiaXi{i}}{h} \right)K_k\left(\frac{\psiaxi + \psiaXi{j}}{h} \right) \right]\\
    {T}_{6,i,j} &:=
    \Eb\left[K_k\left(\frac{\psiaxi + \psiaXi{i}}{h} \right)K_k\left(\frac{\psiaxi - \psiaXi{j}}{h} \right) \right]\\
    {T}_{7,i,j} &:=
    \Eb\left[K_k\left(\frac{\psiaxi + \psiaXi{i}}{h} \right)K_k\left(\frac{\psiaxi + \psiaXi{j}}{h} \right) \right]
\end{align*}
Therefore, because the random variables $\X_1, \dots, \X_n$ are identically distributed, we get
\begin{multline*}
    \Eb \left[ \widehat{\eta}^2 \right]
    = \frac{\Gamma(d/2)^2}{\pi^d n h^{2(k+1)}}
    \big( T_{1,1} + T_{2,1} + 2 T_{3,1}
    + (n - 1) {T}_{4,1,2}
    + (n - 1) {T}_{5,1,2} \\
    + (n - 1) {T}_{6,1,2}
    + (n - 1) {T}_{7,1,2}
    \big),
\end{multline*}

\textbf{Step 1. Removing terms that disappear.}
We start by applying Lemma \ref{lemma:integral_zero_large_n}.
This means that for all $n$ large enough,
\begin{align*}
    \Eb \left[ \widehat{\eta}^2 \right]
    = \frac{\Gamma(d/2)^2}{\pi^d n h^{2(k+1)}}
    \big( T_{1,1} + (n - 1) {T}_{4,1,2}
    \big).
\end{align*}

\textbf{Step 2. Computation of an equivalent of $T_{1,1}$.}
We apply Lemma~\ref{lemma:change_variable_SW}, with $K^2$ instead of $K$, and obtain
\begin{align*}
    {T}_{1,1} &= \frac{\pi^{d/2}h}{\Gamma(d/2)}
    \int_{-\psiaxi/h}^{\infty}K_k^2(-w)
    \frac{g\left(\psi^{-1}_a
    \left(\psiaxi + wh\right)\right)
    \cdot v_1(w)^{(d-2)/2}
    }{\psi_a'\left(\psi^{-1}_a
    \left(\psiaxi + wh\right)\right)} \ dw
\end{align*}
and so,
\begin{align*}
    \frac{\Gamma(d/2)^2}{\pi^d n h^{2(k+1)}}
    {T}_{1,1}
    &= \frac{\Gamma(d/2)}{\pi^{d/2}}
    \frac{1}{nh^{2k+1}} \\
    &\qquad 
    \times \int_{-\psiaxi/h}^{\infty} K_k^2(-w) \frac{g\left(\psi_a^{-1}\left(\psiaxi + wh\right)\right)\cdot v_1(w)^{(d-2)/2}}
    {\psi_a'\left(\psi_a^{-1}\left(\psiaxi + wh\right)\right)}\ dw \\
    &\sim \frac{\Gamma(d/2)}{\pi^{d/2}}\frac{1}{nh^{2k+1}} \gxi \int_{-\infty}^{\infty}K_k^2(w)\ dw,
\end{align*}
as $n \to \infty$, by Bochner's lemma
and remarking that
$\int_{-\infty}^{\infty}K_k^2(w) dw
= \int_{-\infty}^{\infty}K_k^2(-w) dw$.

\medskip

\textbf{Step 3. Computation of $T_{4,1,2}$.}
Since $\X_1, \dots, \X_n$ are independent and identically distributed, we have
\begin{align*}
    {T}_{4,1,2}
    &=
    \Eb\left[K_k\left(\frac{\psiaxi
    - \psiaXi{1}}{h} \right)\right] \cdot
    \Eb\left[K_k\left(\frac{\psiaxi
    - \psiaXi{2}}{h} \right)\right]  \\
    &= \Eb\left[K_k\left(\frac{\psiaxi
    - \psiaXi{1}}{h} \right)\right]^2 \\
    &=\frac{\pi^{d} h^{2(k+1)}}{\Gamma(d/2)^2} \Eb\left[ \widehat{\eta} \right]^2,
\end{align*}
for $n$ large enough by Lemma~\ref{lemma:integral_zero_large_n}.

\medskip

\textbf{Step 4. Combining all previous results.}
\begin{align*}
    \Var\left(\widehat{\eta}\right)
    &= \Eb\left[\widehat{\eta}^2\right]
    - \Eb\left[\widehat{\eta}\right]^2 \\
    &= \frac{\Gamma(d/2)^2}{\pi^d n h^{2(k+1)}}
    \big( T_{1,1} + (n - 1) {T}_{4,1,2}
    \big) 
    - \Eb\left[\widehat{g}_{n,h,a}(\x)\right]^2 \\
    &= \frac{\Gamma(d/2)^2}{\pi^d n h^{2(k+1)}}
    \big( {T}_{1,1} + (n - 1) {T}_{4,1,2}
    \big) - \left( \rho_a^{(k)} \left(\psiaxi\right) - \Bias_{\widehat{\eta}}(\x) \right)^2 \\
    &= \frac{n-1}{n}\left( \rho_a^{(k)} \left(\psiaxi\right) - \Bias_{\widehat{\eta}}(\x) \right)^2
    + \frac{\Gamma(d/2)^2}{\pi^d n h^{2(k+1)}}{T}_{1,1} \\
    &\qquad - \left( \rho_a^{(k)} \left(\psiaxi\right) - \Bias_{\widehat{\eta}}(\x) \right)^2 \\
    &= -\frac{1}{n} \left( \rho_a^{(k)} \left(\psiaxi\right) - \Bias_{\widehat{\eta}}(\x) \right)^2  +
    \frac{\Gamma(d/2)^2}{\pi^d n h^{2(k+1)}}{T}_{1,1}
\end{align*}
Since $\Bias_{\widehat{\eta}}(\x) = o(1)$, we get
\begin{align*}
    \Var\left(\widehat{\eta}\right)
    &= \frac{\Gamma(d/2)^2}{\pi^d n h^{2(k+1)}}{T}_{1,1}
    + O(1/n) \\
    &= \frac{\Gamma(d/2)}{\pi^{d/2}}\frac{1}{nh^{2k+1}} \gxi
    \norm{K_k}_2^2 + o(n^{-1} h^{-(2k+1)} ) + O(n^{-1}) \\
    &\sim \frac{\Gamma(d/2)}{\pi^{d/2}}\frac{1}{nh^{2k+1}} \gxi
    \norm{K_k}_2^2,
\end{align*}
since $h(n) \to 0$ as $n \to \infty$.
So, 
\begin{align*}
    \MSE\left[ \widehat{\eta} \right] &= \left(\mu_m(K_k) \frac{h^{m-k}}{m!} \rho_a^{(m)} \left(\psiaxi\right)\right)^2 + \frac{\Gamma(d/2)}{\pi^{d/2}}\frac{1}{nh^{2k+1}} \gxi
    \norm{K_k}_2^2\\
    & + O\left(h^{2(m-k+1)}\right) + o\left(n^{-1} h^{-(2k+1)} \right) + O\left(n^{-1}\right).
\end{align*}

\end{document}